\documentclass[12pt]{amsart}
\usepackage{geometry}
\usepackage[mathscr]{euscript}
\usepackage{ amssymb }
\usepackage[comma, numbers]{natbib}
\usepackage{bm}
\usepackage{enumerate}
\usepackage[english]{babel}
\usepackage{commath}
\usepackage{tikz-cd}
\usepackage{graphicx}
\usepackage [autostyle, english = american]{csquotes}
\MakeOuterQuote{"}

\theoremstyle{plain}
\newtheorem{theorem}{Theorem}[section]

\theoremstyle{plain}
\newtheorem{lemma}[theorem]{Lemma}

\theoremstyle{plain}
\newtheorem{corollary}[theorem]{Corollary}

\theoremstyle{definition}
\newtheorem{definition}[theorem]{Definition}

\theoremstyle{plain}
\newtheorem{proposition}[theorem]{Proposition}

\theoremstyle{remark}
\newtheorem{remark}[theorem]{Remark}

\theoremstyle{definition}
\newtheorem{example}[theorem]{Example}

\theoremstyle{plain}

\theoremstyle{plain}

\theoremstyle{plain}
\newtheorem{question}[theorem]{Question}

\title{Coarse Proximity and Proximity at Infinity}
\author{Pawel Grzegrzolka}
\address{University of Tennessee, Knoxville, USA} 
\email{pgrzegrz@vols.utk.edu}

\author{Jeremy Siegert}
\address{University of Tennessee, Knoxville, USA} 
\email{jsiegert@vols.utk.edu}
\date{\today} 

\keywords{coarse geometry, metric geometry, coarse topology, coarse proximity, proximity, hyperspace, proximity at infinity\\
published in \textit{Topology and its Applications}, 251:18-46,2019, published version available at https://doi.org/10.1016/j.topol.2018.10.009}

\begin{document}

\begin{abstract}
We define coarse proximity structures, which are an analog of small-scale proximity spaces in the large-scale context. We show that metric spaces induce coarse proximity structures, and we construct a natural small-scale proximity structure, called the proximity at infinity, on the set of equivalence classes of unbounded subsets of an unbounded metric space given by the relation of having finite Hausdorff distance. We show that this construction is functorial. Consequently, the proximity isomorphism type of the proximity at infinity of an unbounded metric space $X$ is a coarse invariant of $X$.
\end{abstract}

\maketitle
\tableofcontents

\section{Introduction}
In classical topology, there are various structures for studying small-scale notions. Of particular interest are uniform structures (see \cite{Isbell} or \cite{willard}) which axiomatize notions of uniform continuity and uniform boundedness, as well as proximity structures which axiomatize the notion of nearness (see \cite{proximityspaces}). In contrast to classical topology, coarse topology investigates the large-scale aspects of spaces, including the large-scale analog of uniform spaces, called coarse spaces (see \cite{Roe} or \cite{alternativedefinition}). Because coarse spaces generalize coarse properties of metric spaces, coarse geometry of metric spaces provides significant insight into the properties of coarse spaces.

Recently, coarse topologists have attempted to define a notion of large-scale proximity. In \cite{Hartmann}, Hartmann defines a binary relation on the power set of a metric space as the negation of asymptotic disjointness. This "closeness" relation is used to construct a uniform space on a set of equivalence classes of certain unbounded subsets of a metric space. In \cite{Honari}, asymptotic resemblance relations are defined, which generalize the notion of the Hausdorff distance between two subsets of a metric space being finite. As shown in \cite{Honari}, asymptotic resemblances "coarsen" many foundational results of proximity structures; however, there are several significant differences between the two notions. For example, asymptotic equivalences are equivalence relations, whereas small-scale proximity relations are not. In other words, an asymptotic resemblance on a set captures when two subsets are "the same at infinity" instead of capturing their coarse closeness. In this paper, we define a coarse analog of proximity spaces which stems from a more direct translation of small-scale proximities into the coarse context. In section \ref{basics}, we review the needed definitions and results related to small-scale proximity spaces and coarse spaces. In section \ref{mainbody}, we introduce the metric coarse proximity relation, and in section \ref{coarseproximityspaces} we generalize this relation to an arbitrary set with a bornology. Section \ref{coarseneighborhoods} is devoted to  coarse neighborhoods and their properties. In section \ref{equivalencerelation}, we show that every coarse proximity induces an equivalence relation on the power set of a given coarse proximity space, which we call a weak asymptotic resemblance. In the case of metric spaces, this equivalence relation is equivalent to the Hausdorff distance between two sets being finite. In section \ref{coarsecategory}, we introduce coarse proximity maps, and we describe the category of coarse proximity spaces whose morphisms are closeness classes of coarse proximity maps. Finally, in section \ref{hyperspace} we show how to define a natural small-scale proximity structure on the set of unbounded subsets of a metric space. We also show how this structure naturally induces a small-scale proximity on the equivalence classes of the weak asymptotic resemblance induced by the metric. We call this space the proximity space at infinity of the metric space and we prove that this construction is functorial. We conclude with some open questions.
\section{Preliminaries}\label{basics}
In an attempt to provide the reader with a self-contained paper, we devote this section to recalling needed definitions and results related to proximity spaces and coarse spaces. An experienced reader may want to skip ahead to section \ref{mainbody} and refer to section \ref{basics} if necessary.

\begin{definition}\label{proximity}
Let $X$ be a set. A \textbf{proximity} on a set $X$ is a relation $\delta$ on the power set of $X$ satisfying the following axioms for all $A,B,C \subseteq X:$
\begin{enumerate}[(i)]
	\item $A\delta B$ implies $B\delta A,$
	\item $(A\cup B)\delta C$ if and only if $A\delta C\text{ or }B\delta C,$
	\item $A\delta B$ implies $A\neq\emptyset\text{ and }B\neq\emptyset,$
	\item $A\bar{\delta}B$ implies that there exists a subset $E$ such that $A\bar{\delta}E$ and $(X\setminus E)\bar{\delta}B,$
	\item $A\cap B\neq\emptyset$ implies $A\delta B,$
\end{enumerate}
where $A\bar{\delta}B$ means "$A\delta B$ is not true." If $A \delta B$, then we say that $A$ is \textbf{close} to (or \textbf{near}) $B.$ Axiom (ii) is called the {\bf union axiom} and (iv) is usually called the {\bf strong axiom}. A {\bf proximity space} is a pair $(X,\delta),$ where $X$ is a set and $\delta$ is proximity on $X$ as defined above.
\end{definition}

\begin{example}
If $(X,d)$ is a metric space, then the proximity relation defined by
\[A\delta B \quad \text{if and only if} \quad d(A,B)=0,\]
where $d(A,B):=\inf\{ d(x,y) \mid x \in A, y \in B \}$, is called the \textbf{metric proximity}.
\end{example}

\begin{definition}
	Given a proximity space $(X,\delta)$ and subsets $A,B\subseteq X,$ we say that $B$ is a {\bf proximal neighborhood} of $A$, denoted $A\ll B,$ if $A\bar{\delta}(X\setminus B)$.
\end{definition}

\begin{definition}
	Given a proximity space $(X,\delta),$ the {\bf induced topology} on $X$ is defined by the closure operator $cl(A)=\{x\in X\mid\{x\}\delta A\}$.
\end{definition}

\begin{definition}
	A function $f:(X,\delta_{1})\rightarrow(Y,\delta_{2})$ is called a {\bf proximity map} if for all $A,B\subseteq X,$ $A\delta_{1} B$ implies $f(A)\delta_{2}f(B)$.
\end{definition}

\begin{remark}
	All proximity maps are continuous with respect to the induced topologies on the domain and codomain.
\end{remark}

\begin{definition}
	Given a set $X$ and two proximities $\delta_{1},\delta_{2}$ on $X,$ we say that $\delta_{1}$ is \textbf{finer} than $\delta_{2}$ (or $\delta_{2}$ is \textbf{coarser} than $\delta_{1}$), denoted $\delta_{1}>\delta_{2}$, if $A\delta_{1}B$ implies $A\delta_{2}B$.
\end{definition}

The following result is from \cite{proximityspaces}:

\begin{proposition}\label{proximityondomain}
	Given a function $f:X\rightarrow (Y,\delta_{2})$, the coarsest proximity $\delta_{0}$ on $X$ for which $f$ is a proximity map is defined by
	\[A\bar{\delta}_{0}B\text{ if and only if there is a }C\subseteq Y\text{ such that }f(A)\bar{\delta}_{2}(Y\setminus C)\text{ and }f^{-1}(C)\subseteq(X\setminus B)\]
\end{proposition}

\begin{definition}\label{surjectiveproximity}
	Let $f$ be a surjective function from a proximity space $(X,\delta)$ onto a set $Y$. The \textbf{quotient proximity} is the finest proximity on $Y$ for which $f$ is a proximity map.
\end{definition}

In \cite{friedler}, it is shown that such a proximity always exists. For a detailed description of quotient proximities we refer the reader to \cite{friedler}. An important property of quotient proximities that will be used in section \ref{hyperspace} is the following:

\begin{proposition}\label{quotientproximityproperty}
	Let $(X,\delta_{1})$ be a proximity, $f:X\rightarrow Y$ a surjective function, and $g:Y\rightarrow(Z,\delta_{3})$ a function. If $\delta_{2}$ is the quotient proximity on $Y$ induced by $f,$ then $g\circ f$ is a proximity map if and only if $g$ is a proximity map from $(Y,\delta_{2})$ to $(Z,\delta_{3})$.
		\begin{center}
		\begin{tikzcd}[column sep=large]
			(X,\delta_{1}) \arrow{dd}{f} \arrow{rdd}{g \circ f} 
			&  \\
			& \\
			Y \arrow{r}{g} & (Z, \delta_{3})
		\end{tikzcd}
	\end{center}
\end{proposition}
\begin{proof}
	It is clear that if $g$ is a proximity map, then so is $g\circ f$. To prove the converse, assume that $g\circ f$ is a proximity map. Consider the proximity $\delta_{g}$ induced on the set $Y$ by $g$ as defined in Proposition \ref{proximityondomain}. We will show that $\pi:(X,\delta_{1})\rightarrow(Y,\delta_{g})$ is a proximity mapping. Let $A,B\subseteq Y$ be such that $A\bar{\delta}_{g}B$. Then there is a $C\subseteq Z$ such that $g(A)\bar{\delta}_{3}(Z\setminus C)$. By the strong axiom there is then a $D\subseteq Z$ such that $g(A)\bar{\delta}_{3}D$ and $(Z\setminus D)\bar{\delta}_{3}(Z\setminus C)$. As proven in \cite{proximityspaces}, the set $E=g^{-1}(D)$ is a set such that $A\bar{\delta}_{g}E$ and $(Y\setminus E)\bar{\delta}_{g}B$. If $\pi^{-1}(A){\delta}_{1}\pi^{-1}(E),$ then we have
	\[(g\circ\pi)(\pi^{-1}(A))\delta_{3}(g\circ\pi)(\pi^{-1}(E)).\]
	However, note that
	\[(g\circ\pi)(\pi^{-1}(A))=g(A)\text{ and }(g\circ\pi)(\pi^{-1}(E))=g(E)= D,\]
	which would imply that $g(A)\delta_{3}D,$ which is a contradiction. Then, because $\pi^{-1}(B)\subseteq\pi^{-1}(E)$ we have that $\pi^{-1}(A)\bar{\delta}_{1}\pi^{-1}(B)$, which establishes that $\pi$ is a proximity map when $Y$ is equipped with the proximity $\delta_{g}$. By the definition of the quotient proximity we must then have that $\delta_{2}$ is finer than $\delta_{g}$. Now assume towards a contradiction that $g:(Y,\delta_{2})\rightarrow(Z,\delta_{3})$ is not a proximity mapping. Then there are subsets $A,B\subseteq Y$ such that $A\delta_{2}B$ and $g(A)\bar{\delta}_{3}g(B)$. However, because $\delta_{2}$ is finer than $\delta_{g}$ we have that $A\delta_{g}B$ and because $g$ is a proximity mapping when $Y$ is equipped with $\delta_{g}$ we have that $g(A)\delta_{3}g(B)$, which is a contradiction. Thus, we must have that $g$ is a proximity mapping.
\end{proof}

\begin{definition}
A \textbf{coarse structure} on a set $X$ is a collection $\mathcal{E}$ of subsets of $X \times X,$ called \textbf{controlled sets} or \textbf{entourages}, such that the following are satisfied:
\begin{enumerate}
\item $\triangle \in \mathcal{E}$, where $\triangle:=\{(x, x) \mid x \in X\},$
\item if $E \in \mathcal{E}$ and $B \subseteq E,$ then $B \in \mathcal{E},$
\item if $E \in \mathcal{E},$ then $E^{-1} \in \mathcal{E},$ where $E^{-1}:=\{(x,y) \mid (y,x) \in E\},$
\item if $E \in \mathcal{E}$ and $F \in \mathcal{E},$ then $E \cup F \in \mathcal{E},$
\item if $E \in \mathcal{E}$ and $F \in \mathcal{E},$ then $E \circ F \in \mathcal{E},$ where $E \circ F:=\{(x,y) \mid \exists \, z \in X \text{ such that } (x,z) \in E, (z,y) \in F\}.$
\end{enumerate}
A set $X$ endowed with a coarse structure $\mathcal{E}$ is called a \textbf{coarse space}.
\end{definition}

\begin{definition}
If $(X, \mathcal{E})$ is a coarse space, $A$ a subset of $X,$ and $E$ a controlled set, then we define
\[E[A]=\{x \in X \mid \exists \, a \in A \text { such that } (x,a) \in E\}.\]
\end{definition}

\begin{definition}
Let $(X,d_1)$ and $(X, d_2)$ be metric spaces. Let $f: X \to Y$ be a function. Then
\begin{enumerate}
\item $f$ is called \textbf{proper} if the inverse images (under $f$) of bounded sets in $Y$ are bounded in $X,$
\item $f$ is called (uniformly) \textbf{bornologous} if uniformly bounded families of sets are sent to uniformly bounded families, i.e., for all $R>0$ there exists $S>0$ such that
\[d_1(x_1,x_2)<R \implies d_2(f(x_1), f(x_2))<S,\]
\item $f$ is called \textbf{coarse} if it is proper and bornologous.
\end{enumerate}
\end{definition}

\begin{definition}\label{coarselyclose}
Let $X$ be a set and $(Y,d)$ a metric space. Two functions $f,g:X \to Y$ are {\bf coarsely close} if there exists $C>0$ such that for all $x\in X,$
\[d(f(x),g(x)) < C.\]
\end{definition}

\section{Metric Coarse Proximity}\label{mainbody}

In this section, we define a relation on the power set of a metric space. This relation captures the "closeness at infinity" of subsets of a metric space. We also prove several properties of this relation.

\begin{definition}
Let $(X,d)$ be a metric space. Let $A$ and $B$ be subsets of $X.$ Then the distance between $A$ and $B$ is defined by
\[d(A,B):=\inf\{ d(x,y) \mid x \in A, y \in B \},\]
and the \textbf{Hausdorff distance} between $A$ and $B$ is defined by
\[d_{H}(A,B)=\inf\{\epsilon \mid A\subseteq \mathscr{B}(B,\epsilon)\text{ and }B\subseteq \mathscr{B}(A,\epsilon)\},\]
where $\mathscr{B}(A,\epsilon)$ is the open ball of radius $\epsilon$ about $A$ (i.e., $\mathscr{B}(A,\epsilon)=\bigcup_{x\in A}\mathscr{B}(x,\epsilon)$) and $\mathscr{B}(B,\epsilon)$ is the open ball of radius $\epsilon$ about $B.$

\end{definition}

\begin{remark}
Recall that by convention the infimum of the empty set is $\infty$. Thus, if either $A$ or $B$ is the empty set, then $d(A,B)= \infty.$
\end{remark}

\begin{definition}\label{metriccoarseproximity}
Let $(X,d)$ be a metric space. Let $A$ and $B$ be subsets of $X.$ We say that $A$ and $B$ are \textbf{coarsely close}, denoted $A {\bf b} B$, if there exists $\epsilon < \infty$ such that for all bounded sets $D$, $d(A \setminus D, B\setminus D)< \epsilon.$
\end{definition}

\begin{remark}
If $(X,d)$ is a metric space and $A$ is a bounded subset of $X$, then $A$ is not coarsely close to any subset of $X.$ Consequently, if $X$ is bounded, the relation is empty.
\end{remark}

\begin{proposition}
Let $(X,d)$ be a metric space. Let $A$ and $B$ be subsets of $X.$ Then the following are equivalent:
\begin{enumerate}
\item there exists $\epsilon < \infty$ such that for all bounded sets $D$, $d(A \setminus D, B\setminus D)< \epsilon,$
\item there exists $\epsilon < \infty$ such that for all bounded sets $D$, there exists $a \in (A \setminus D)$ and $b \in (B \setminus D)$ such that $d(a,b) < \epsilon,$
\item there exist unbounded sets $A_{1}\subseteq A,\,B_{1}\subseteq B$ such that $d_{H}(A_{1},B_{1})<\infty.$
\end{enumerate}
\end{proposition}

\begin{proof}
Exercise.
\end{proof}
 
 For the remainder of the paper, all of the equivalent conditions for coarse closeness in a metric space will be used interchangeably without explicit mention. The reader is encouraged to compare the equivalent conditions from the above proposition with the notion of asymptotic disjointness given in \cite{Honari}, \cite{Bell}, or Definition \ref{asymptoticdisjointness}.

\begin{theorem}\label{theorem1}
Let $(X,d)$ be a metric space. Let $A, B$ and $C$ be subsets of $X.$ Then the following are true:
\begin{enumerate}[(i)]
	\item $A{\bf b}B$ implies $B{\bf b}A,$
	\item $(A \cup B){\bf b}C$ if and only if $A{\bf b}C$ or $B{\bf b}C,$
	\item $A{\bf b}B$ implies $A$ is unbounded and $B$ is unbounded,
	\item $A\bar{\bf b}B$ implies that there exists a set $E$ such that $A\bar{\bf b}E$ and $(X\setminus E)\bar{\bf b}B$.
	\item $A\cap B$ not bounded implies $A {\bf b} B,$
\end{enumerate}
where $A\bar{ {\bf b}}B$ means "$A{\bf b} B$" is not true.
\end{theorem}

\begin{proof}
Properties (i), (iii), and (v) are clear. We will show (ii) and (iv).

The backward direction of (ii) is trivial. To show the forward direction, assume $(A \cup B){\bf b}C$ and for contradiction assume that $A\bar{{\bf b}}C$ and $B\bar{{\bf b}}C.$ Since $(A \cup B){\bf b}C,$ there exists $\epsilon < \infty$ such that for all bounded sets $D,$
\[d((A\cup B) \setminus D, C \setminus D)< \epsilon.\]
Since $A\bar{{\bf b}}C$ and $B\bar{{\bf b}}C$, there exist bounded sets $D_1$ and $D_2$ such that
\[d(A \setminus D_1, C \setminus D_1)> \epsilon \quad \text{ and } \quad d(B \setminus D_2, C \setminus D_2)> \epsilon.\]
Let $D:= D_1 \cup D_2.$ Then notice that $D$ is bounded and
\[d(A \setminus D, C \setminus D)> \epsilon \quad \text{ and } \quad d(B \setminus D, C \setminus D)> \epsilon,\]
which implies that 
\[d((A\cup B) \setminus D, C \setminus D) > \epsilon,\]
a contradiction.

To prove (iv), notice that if $A$ is bounded, then the set $E:=(X\setminus A)$ has the desired properties. If $B$ is bounded, then $E:=B$ has the desired properties. Thus, assume that both $A$ and $B$ are unbounded and $A\bar{\bf b}B$. Then for every $n\in\mathbb{N}$ there is a bounded set $D_n \subseteq X$ such that $d(A\setminus D_n,B\setminus D_n)>n^2.$ Fix some $x_0 \in X.$ Since any bounded set is contained in some large ball centered at $x_0$, without loss of generality assume that each $D_n$ is a ball centered at $x_0$ with radius $r_n,$ i.e., $D_n=\mathscr{B}(x_0, r_n).$ Additionally, one can assume that the radii are strictly increasing as $n \to \infty$ and that they take integer values. We can even assume that for each $n,$ we have $r_n-r_{n-1}>n+1.$ 

For each $n,$ define
\begin{equation*}
\begin{split}
E_{0} & :=B,\\
E_{n} & := \mathscr{B}(B, n) \setminus \mathscr{B}(x_0, r_n),\\
E & :=\bigcup_{n \geq 0} E_{n}.
\end{split}
\end{equation*}
Notice that this definition implies that $d(X \setminus D_n,D_{n-1})>n$ for all $n>1.$ Notice that $E$ is unbounded, since $B \subseteq E$. We will show that that $A\bar{\bf b}E$ and $(X\setminus E)\bar{\bf b}B.$ 

\begin{figure}
\caption{Construction of $E$}
\includegraphics[scale=0.3]{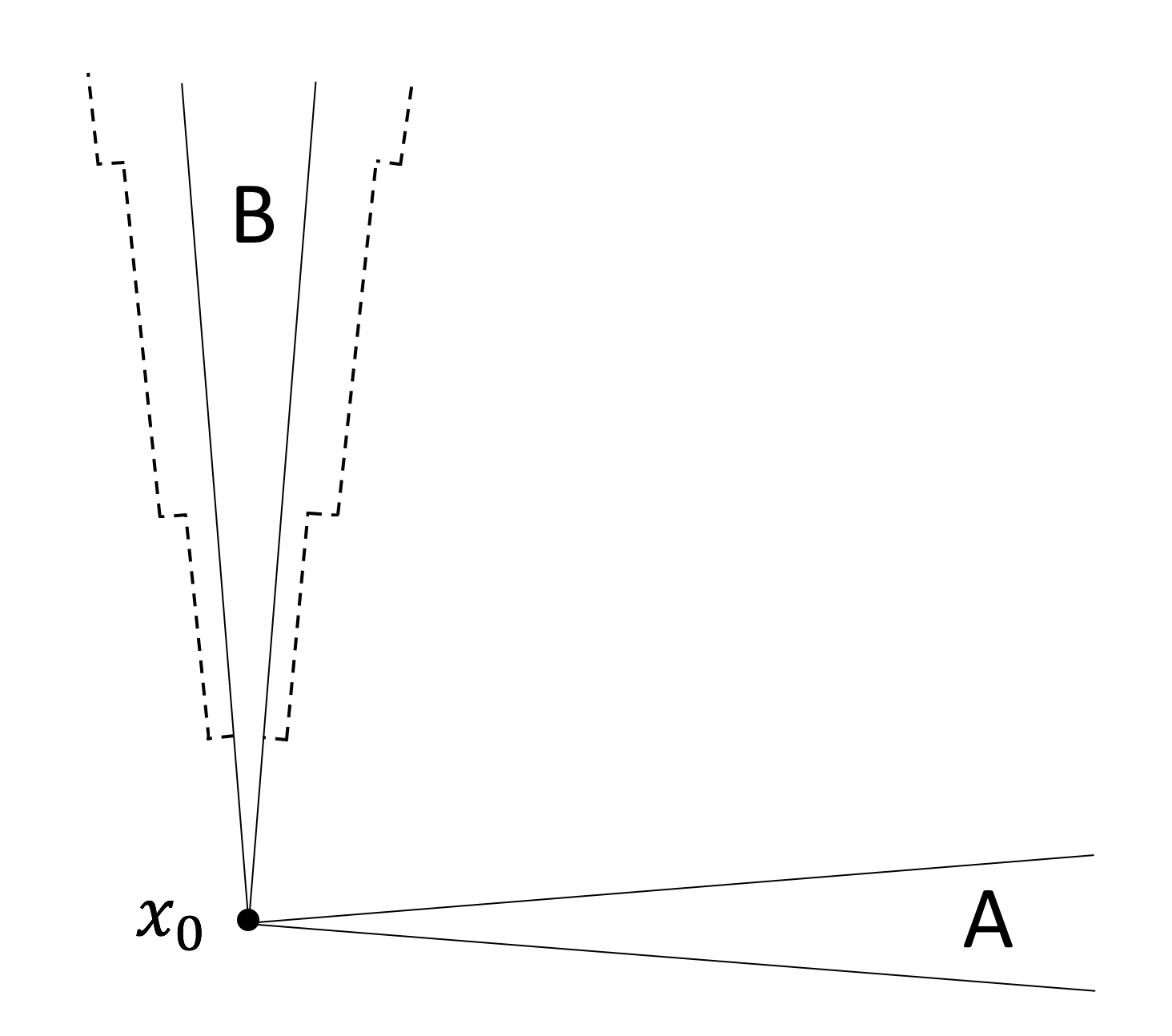}
\centering
\label{Coarseneighborhood}
\end{figure}

First assume that $A{\bf b}E.$ Then there exists $\epsilon < \infty$ such that for all $n \in \mathbb{N},$ there exists $x_n \in E \setminus D_n$ and $a_n \in A \setminus D_n$ such that $d(x_n,a_n)<\epsilon.$  Find $n$ so large that it satisfies the following inequalities: 
\begin{equation*}
\begin{split}
n & >1,\\
n & > \epsilon, \\
(n-1)^2 & > n+\epsilon.
\end{split}
\end{equation*}
Notice that the above inequalities are satisfied for any $k\geq n.$ Let $k$ be the largest integer such that $x_n \notin D_k$. Clearly $k \geq n$ and $x_n \in D_{k+1}.$ Consequently, $x_n \in D_{m}$ for all $m \geq k+1.$ This implies that $x_n \notin E_m$ for all $m \geq k+1.$ Therefore, since $x_n \in E,$ there exists $b \in B$ such that $d(b,x_n)<k.$ Notice that this also implies that $b \notin D_{k-1}$ (because $x_n \in (X \setminus D_k)$ and $d(X \setminus D_k,D_{k-1})>k).$ Similarly, $d(x_n,a_n)< \epsilon $ implies that $a_n \notin D_{k-1}$ (because $d(X \setminus D_k,D_{k-1})>k>\epsilon).$ Thus we have $a_n \in (A \setminus D_{k-1}), b \in (B \setminus D_{k-1}),$ and 
 \[d(b, a_n) \leq d(b, x_n)+d(x_n, a_n)< k + \epsilon<(k-1)^2.\]
 contradicting $d(A\setminus D_{k-1},B\setminus D_{k-1})>(k-1)^2.$ Thus, it has to be that $A\bar{\bf b}E.$

 To show that $(X\setminus E)\bar{\bf b}B,$ for contradiction assume that $(X\setminus E){\bf b}B.$ Then there exists $\epsilon < \infty$ such that for all $n \in \mathbb{N},$ there exists $x_n \in (X \setminus E) \setminus \mathscr{B}(x_0, r_n)$ and $b_n \in B \setminus \mathscr{B}(x_0, r_n)$ such that $d(x_n,b_n)<\epsilon.$ Choose $n$ large so that $\epsilon < n.$ Then $x_n \notin \mathscr{B}(x_0, r_n)$ and $d(x_n, b_n) < \epsilon < n.$ In other words, $x_n \in E_n$, contradiciting the fact that $x_n \notin E.$ Therefore, it has to be that $(X\setminus E)\bar{\bf b}B.$
\end{proof}

\begin{remark}
The specific construction of $E$ in the above proof will be utilized in section \ref{hyperspace}. However, to prove the strong axiom, one could choose $E:= \{x \in X \mid d(x, B) \leq d(x,A) \}$ \footnote{The authors are grateful to Thomas Weighill for suggesting this alternative construction.}. We leave the verification of this fact as an exercise for the reader.
\end{remark}

\begin{remark}\label{remark2}
One can easily show that if $B$ is unbounded, then the set $E$ from axiom (iv) has to be unbounded as well.
\end{remark}

\begin{remark}
Notice that, unlike asymptotic resemblance (for the definition, see \cite{Honari} or definition \ref{asymptoticresemblancedefinition}), the relation in Definition \ref{metriccoarseproximity} is not an equivalence relation. To see that, let $X= \mathbb{R}^2,$ $A$ the positive $x$-axis, $B$ the first quadrant, and $C$ the positive $y$-axis. Then $A {\bf b} B,$ $B {\bf b} C,$ but $A \bar{{\bf b}} C.$
\end{remark}

\section{Coarse Proximity Spaces}\label{coarseproximityspaces}

In this section, we generalize the coarse proximity relation from section \ref{mainbody} to an arbitrary set $X.$ We also explore several properties of this relation.

In order to generalize the coarse proximity relation from section \ref{mainbody} to an arbitrary set $X,$ we need to have a notion of a subset being "bounded." Thus, we recall the following definition:

\begin{definition}
	A {\bf bornology} $\mathcal{B}$ on a set $X$ is a family of subsets of $X$ satisfying:
	\begin{enumerate}[(i)]
		\item $\{x\}\in\mathcal{B}$ for all $x\in X,$
		\item $A\in\mathcal{B}$ and $B\subseteq A$ implies $B\in\mathcal{B},$
		\item If $A,B\in\mathcal{B},$ then $A\cup B\in\mathcal{B}.$
	\end{enumerate}
Elements of $\mathcal{B}$ are called {\bf bounded} and subsets of $X$ not in $\mathcal{B}$ are called {\bf unbounded}. If $X \notin \mathcal{B},$ then we call the bornology \textbf{proper}.
\end{definition}

\begin{definition}
Let $X$ be a set equipped with a bornology $\mathcal{B}$. A \textbf{coarse proximity} on a set $X$ is a relation ${\bf b}$ on the power set of $X$ satisfying the following axioms for all $A,B,C \subseteq X:$

\begin{enumerate}[(i)]
	\item $A{\bf b}B$ implies $B{\bf b}A,$
	\item $(A \cup B){\bf b}C$ if and only if $A{\bf b}C$ or $B{\bf b}C,$
	\item $A{\bf b}B$ implies $A \notin \mathcal{B}$ and $B \notin \mathcal{B},$
	\item $A\bar{\bf b}B$ implies that there exists a subset $E$ such that $A\bar{\bf b}E$ and $(X\setminus E)\bar{\bf b}B,$
	\item $A\cap B \notin \mathcal{B}$ implies $A {\bf b} B.$
\end{enumerate}
where $A\bar{ {\bf b}}B$ means "$A{\bf b} B$ is not true." If $A {\bf b} B$, then we say that $A$ is \textbf{coarsely close} to (or \textbf{coarsely near}) $B.$ Axiom (iv) will be called the \textbf{strong axiom}. A triple $(X,\mathcal{B},{\bf b})$ where $X$ is a set, $\mathcal{B}$ is a bornology on $X$, and ${\bf b}$ is a coarse proximity relation on $X,$ is called a {\bf coarse proximity space}.
\end{definition}

The reader is encouraged to compare the above axioms with the axioms of a (small-scale) proximity given in definition \ref{proximity}. For the remainder of the paper, the use of axiom (i) will not be explicitly mentioned.

\begin{example}
Let $(X,d)$ be a metric space, $\mathcal{B}_{d}$ the collection of all bounded sets of $X$ with respect to the metric $d$, and ${\bf b}_d$ the relation defined in \ref{metriccoarseproximity}. Then by theorem \ref{theorem1}, this relation is a coarse proximity on $X$. We call this relation the \textbf{metric coarse proximity} and the associated space $(X, \mathcal{B}_{d}, {\bf b}_d)$  the \textbf{metric coarse proximity space}.
\end{example}

\begin{example}
Let $X$ be a set with any bornology $\mathcal{B}.$ For any subsets $A$ and $B$ of $X,$ define 
\[A{\bf b}B \quad \text{ if } \quad A \cap B \notin \mathcal{B}.\]
Then this relation is a coarse proximity on $X,$ called the \textbf{discrete coarse proximity}.
\end{example}
\begin{proof}
All the axioms are clear besides axiom (iv). To show axiom (iv), set $E=B.$
\end{proof}

\begin{example}
Let $X$ be a set with any bornology $\mathcal{B}.$ For any subsets $A$ and $B$ of $X,$ define 
\[A{\bf b}B \quad \text{ if } \quad A,B  \notin \mathcal{B}.\]
Then this relation is a coarse proximity on $X,$ called the \textbf{indiscrete coarse proximity}.
\end{example}
\begin{proof}
All the axioms are clear besides axiom (iv). To show axiom (iv), assume $A \bar{\bf b} B.$ If $A \in \mathcal{B},$ let $E=X \setminus A.$ If $B \in \mathcal{B},$ let $E=B.$
\end{proof}

\begin{remark}
Notice that if $\mathcal{B}$ is not a proper bornology on a set $X$ (i.e. there are no unbounded sets), then the coarse proximity relation is empty.
\end{remark}

\begin{remark}\label{remark1}
 If $A\bar{\bf b}B$, then the set $E$ from the strong axiom contains $B$ up to some bounded set, i.e. $B \setminus E \in \mathcal{B}$. For if that is not the case, then $B  \cap (X \setminus E) \notin \mathcal{B},$ which by axiom (v) implies that $(X\setminus E){\bf b}B,$ a contradiction to the definition of $E.$ In particular, if $B \notin \mathcal{B},$ then $E$ has to be unbounded.
\end{remark}

\begin{lemma}\label{lemma1}
Let $(X,\mathcal{B},{\bf b})$ be a coarse proximity space. Let $A, B, C,$ and $D$ be subsets of $X.$ If $A \subseteq C, B \subseteq D,$ and $A {\bf b} B$, then $C {\bm b}D.$ In particular, $X$ is coarsely near every unbounded subset.
\end{lemma}
\begin{proof}
Notice that $A \cup C = C.$ Thus, by axiom (ii), $C {\bf b} B.$ Since $B \cup D = D,$ axiom (ii) implies that $C{ \bf b}D.$
\end{proof}

\begin{remark}\label{remark11}
The above lemma implies that if $A \subseteq C, B \subseteq D,$ and $C \bar{{\bm b}}D,$ then $A \bar{{\bm b}}B.$
\end{remark}

\begin{proposition}\label{proposition2}
Let $(X,\mathcal{B},{\bf b})$ be a coarse proximity space. Let $A$ and $B$ be subsets of $X.$ Then $A{\bf b}B$ if and only if for all $D_1, D_2 \in \mathcal{B},$ $(A\setminus D_1){\bf b}(B\setminus D_2),$
\end{proposition}

\begin{proof}
The converse direction follows from Lemma \ref{lemma1}. To prove the forward direction, assume $A{\bf b}B$ and let $D_1, D_2 \in \mathcal{B}$ be arbitrary. For contradiction, assume $(A\setminus D_1)\bar{{\bf b}}(B\setminus D_2),$ Then notice that since $D_1$ is bounded, $D_1 \bar{\bf b}(B\setminus D_2),$ so by axiom (ii), $A\bar{\bf b}(B\setminus D_2)$. Similarly, $A\bar{{\bf b}} D_2,$ which again by axiom (ii) gives us $A\bar{{\bf b}}B$, a contradiction. Thus, $(A\setminus D_1)\bar{{\bf b}}(B\setminus D_2).$
\end{proof}

\begin{remark}
Notice that the property from proposition \ref{proposition2} is a large scale equivalent of the trivial property of a small scale proximity, namely
\[A\delta B \text{ if and only if } (B \setminus \emptyset) \delta (A \setminus \emptyset).\]
\end{remark}

\begin{proposition}\label{converse_of_strong_axiom}
Let $(X,\mathcal{B},{\bf b})$ be a coarse proximity space. Let $A$ and $B$ be subsets of $X$. Then the converse of the strong axiom holds, i.e., if there exists $E \subseteq X$ such that $A\bar{\bf b}E$ and $(X\setminus E)\bar{\bf b}B,$ then $A\bar{\bf b}B.$
\end{proposition}

\begin{proof}
Assume that there exists $E \subseteq X$ such that $A\bar{\bf b}E$ and $(X\setminus E)\bar{\bf b}B.$ By the proof of remark \ref{remark1}, there exists a bounded set $D$ such that $(B \setminus D) \subseteq E.$ By lemma \ref{lemma1}, this implies that $A \bar{\bf b} (B \setminus D).$ Also, by axiom (iii) we have that $A\bar{\bf b} D$, and thus by axiom (ii), $A\bar{\bf b}B.$
\end{proof}

\section{Coarse Neighborhoods}\label{coarseneighborhoods}

In this section, we introduce the definition of a coarse neighborhood and explore several of its basic properties. We show that if $X$ is a metric space, then coarse neighborhoods coincide with asymptotic neighborhoods defined in \cite{Bell} and that coarse maps copreserve coarse neighborhoods.

\begin{definition}
	Let $(X,\mathcal{B},{\bf b})$ be a coarse proximity space. Given subsets $A,B\subseteq X,$ we say that $B$ is a {\bf coarse neighborhood} of $A,$ denoted $A\ll B,$ if $A\bar{\bf b}(X\setminus B)$. 
\end{definition}

Let us now explore a few basic properties of coarse neighborhoods.

\begin{proposition}\label{propertiesofcoarseneighborhoods}
Let $(X,\mathcal{B},{\bf b})$ be a coarse proximity space. Let $A, B,$ and $C$ be subsets of $X.$ Then the following are true:
\begin{enumerate}[(i)]
\item if $A \in \mathcal{B},$ then $A\ll E$ for any $E \subseteq X,$
\item if $A\ll B,$ then $B$ contains $A$ up to some bounded set, i.e., $(A \setminus B) \in \mathcal{B},$
\item if $A \subseteq B$ and $B\ll C,$ then $A\ll C,$
\item if $(A \setminus B) \in \mathcal{B}$ and  $B\ll C,$ then $A\ll C,$
\item if $A\ll B$ and $B\ll C,$ then $A\ll C,$
\item if $A\ll B$ and $A\ll (X \setminus B),$ then $A \in \mathcal{B},$
\item $A\ll B$ if and only if $(X \setminus B)\ll (X \setminus A).$
\end{enumerate}
\end{proposition}

\begin{proof}
Property (i) is trivial. To show (ii), notice that if $A \cap (X \setminus B) \notin \mathcal{B},$ then $A{\bf b} (X \setminus B),$ a contradiction to $A\ll B.$ (iii) follows from Remark \ref{remark11} and (iv) follows from (iii) and Propostion \ref{proposition2}. To see (v), notice that $A\ll B$ and property (ii) imply $(A \setminus B) \in \mathcal{B}.$ Thus, property (iv) shows $A\ll C.$ To see (vi), notice that $A\ll B$ and $A\ll (X \setminus B)$ imply $(A \setminus B) \in \mathcal{B},$ and $(A \setminus (X\setminus B)) \in \mathcal{B},$ respectively. Thus, $A=(A \setminus B) \cup (A \setminus (X\setminus B)) \in \mathcal{B}.$ Finally, to see (vii), notice that 
\begin{equation*}
\begin{split}
A\ll B & \Longleftrightarrow A\bar{\bf b} (X \setminus B)\\
& \Longleftrightarrow (X \setminus B)\bar{\bf b}A\\
& \Longleftrightarrow (X \setminus B)\bar{\bf b} (X \setminus (X \setminus A))\\
& \Longleftrightarrow (X \setminus B)\ll (X \setminus A). \qedhere
\end{split}
\end{equation*}
\end{proof}

\begin{remark}
Notice that (ii) in Proposition \ref{propertiesofcoarseneighborhoods} implies that if $A$ is unbounded, then so is its coarse neighborhood. Also, property (iv) implies that if $A\ll B$ and $D$ is bounded, then $(A \cup D)\ll B.$
\end{remark}

\begin{remark}\label{strongaxiomneighborhoods}
Notice that by using coarse neighborhoods, the strong axiom can be translated to: $A\bar{\bf b}B$ implies that there exists a subset $E$ such that $A\ll (X \setminus E)$ and $B\ll E.$
\end{remark}

The following proposition characterizes the strong axiom in terms of coarse neighborhoods.

\begin{proposition}\label{intermediateneighborhood}
Let $(X,\mathcal{B})$ be a set with a bornology, and let ${\bf b}$ be a relation on $2^{X}$ satisfying axioms (i),(ii),(iii), and (v) of a coarse proximity. Define, for subsets $A,B\subseteq X$, $A\ll B\iff A\bar{ {\bf b}}(X\setminus B)$. Then for subsets $A,B\subseteq X$ the following are equivalent:
\begin{enumerate}[(i)]
\item $A\bar{\bf b}B$ implies that there exists a subset $E$ such that $A\bar{\bf b}E$ and $(X\setminus E)\bar{\bf b}B,$
\item $A\ll B$ implies that there exists $C \subseteq X$ such that $A\ll C \ll B.$
\end{enumerate}
\end{proposition}

\begin{proof}
($(i) \implies (ii)$) Assume $A\ll B.$ Since $A \bar{\bf b} (X \setminus B),$ by the strong axiom there exists $E \subseteq X$ such that $A \bar{\bf b} E$ and $(X \setminus E) \bar{\bf b} (X \setminus B).$ In other words, we have that $A \bar{\bf b} (X \setminus (X \setminus E))$ and  $(X \setminus E) \bar{\bf b} (X \setminus B),$ i.e., $A\ll (X \setminus E) \ll B.$ Setting $C:=(X \setminus E)$ gives the desired result.\\
($(ii) \implies (i)$) Assume $A\bar{\bf b}B.$ This can be written as $A\bar{\bf b}(X \setminus (X \setminus B)),$ i.e., $A\ll (X \setminus B).$ Therefore, there exists $C \subseteq X$ such that $A\ll C\ll (X \setminus B),$ i.e. $A \bar{\bf b} (X \setminus C)$ and $C \bar{\bf b} (X \setminus (X \setminus B)).$ Let $E=X \setminus C.$ Then $A \bar{\bf b} E$ and $(X\setminus E) \bar{\bf b} B.$
\end{proof}

Now we will investigate the relationship between coarse neighborhoods, asymptotic neighborhoods, and asymptotic disjointness. In particular, we will show that in the case of metric spaces, coarse neighborhoods and asymptotic neighborhoods coincide. Recall the following definitions from \cite{Bell}:

\begin{definition}
	In a metric space $(X,d),$ a subset $B\subseteq X$ is an \textbf{asymptotic neighborhood} of a set $A\subseteq X$ if there exists $x_{0}\in X$ such that
	 \[\lim_{r \to \infty}d(A\setminus \mathscr{B}(x_{0},r),X\setminus B)=\infty.\]
\end{definition}

\begin{definition}\label{asymptoticdisjointness}
	In a metric space $(X,d),$ two subsets $A,B\subseteq X$ are said to be \textbf{asymptotically disjoint} if for some (and hence every) point $x_{0}\in X$ one has
	\[\lim_{r\to\infty}d(A\setminus \mathscr{B}(x_{0},r),B\setminus \mathscr{B}(x_{0},r))=\infty.\]
\end{definition}

The following result follows directly from definitions.

\begin{proposition}\label{asymptoticdisjointnessproposition}
	Let $(X,d)$ be a metric space with the corresponding metric coarse proximity ${\bf b}_{d}.$ Then $A$ and $B$ are asymptotically disjoint in the sense of Definition \ref{asymptoticdisjointness} if and only if $A\bar{ {\bf b}}B$.
\end{proposition}

To compare asymptotic neighborhoods and coarse neighborhoods, we need the following lemma:

\begin{lemma}\label{easierdisjointness}
	Let $(X,d)$ be a metric space and let $A,B\subseteq X.$ Then $A$ and $B$ are asymptotically disjoint if and only if for every $n\in\mathbb{N}$ there is a bounded set $C$ such that $d(A\setminus C,B)>n$.
\end{lemma}
\begin{proof}
	If either $A$ or $B$ is bounded, then the result is trivial. Thus, assume that $A$ and $B$ are unbounded. The reverse direction is trivial. Assume that $A$ and $B$ are asymptotically disjoint and assume towards a contradiction that $n\in\mathbb{N}$ is such that for all bounded $C\subseteq X,$ $d(A\setminus C,B)\leq n$. Thus, for every such bounded set $C$ there is a pair $(x_{C},y_{C})\in A\times B$ such that $x_{C}\notin C$ and $d(x_{C},y_{C})\leq n$. Since $A$ and $B$ are asymptotically disjoint, there is a bounded set $D$ such that $d(A\setminus D,B\setminus D)>n$. Without loss of generality we can assume that $D=\mathscr{B}(x_0, r)$ for some $x_0 \in X$ and some radius $r.$ Thus, for any $r'>r,$ if $C=\mathscr{B}(x_0, r')$, then $x_C \notin C$ and $y_C \in D.$ In particular, if $r'>r+n,$ then we have $x_C \notin \mathscr{B}(x_0, r'), y_C \in \mathscr{B}(x_0, r)$ and $d(x_{C},y_{C})\leq n,$ a contradiction.
\end{proof}

\begin{proposition}
Given a metric space $(X,d)$ and subsets $A,B\subseteq X,$ $B$ is an asymptotic neighborhood of $A$ if and only if $B$ is a coarse neighborhood of $A$ with respect to the metric coarse proximity ${\bf b}_d.$
\end{proposition}

\begin{proof}
Assume that $B$ is an asymptotic neighborhood of $A.$ Then there exists $x_{0}\in X$ such that $\lim_{r \to \infty}d(A\setminus \mathscr{B}(x_{0},r),X\setminus B)=\infty.$ For contradiction, assume that $A{\bf b}(X\setminus B).$ Then there exists $\epsilon < \infty$ such that for any $r$ we can find $x \in A\setminus \mathscr{B}(x_{0},r)$ and $y\in (X\setminus B) \setminus \mathscr{B}(x_{0},r)$ with the property that $d(x,y)< \epsilon.$ In particular, we can find $x \in A\setminus \mathscr{B}(x_{0},r)$ and $y\in (X\setminus B)$ such that $d(x,y)< \epsilon,$ contradicting the fact that $\lim_{r \to \infty}d(A\setminus \mathscr{B}(x_{0},r),X\setminus B)=\infty.$ The converse follows from Proposition \ref{asymptoticdisjointnessproposition} and Lemma \ref{easierdisjointness}.
\end{proof}

The following proposition shows that coarse maps copreserve asymptotic neighborhoods.

\begin{proposition}\label{copreservingneighborhoods}
	Let $(X,d_{1}),(Y,d_{2})$ be metric spaces and $h:X\rightarrow Y$ a coarse map. If $A,B\subseteq Y$ such that $A\ll B$ with respect to the metric coarse proximity structure induced by $d_2,$ then $h^{-1}(A)\ll h^{-1}(B)$ with respect to the metric coarse proximity structure on $X$ induced by $d_1.$
\end{proposition}
\begin{proof}
If $A$ is bounded, then since $h$ is a coarse map, $h^{-1}(A)$ is bounded. By Proposition \ref{propertiesofcoarseneighborhoods}, this implies that any set is a coarse neighborhood of $A.$ In particular, $h^{-1}(A)\ll h^{-1}(B).$ So let us assume that $A$ is unbounded. Let $x_{0}\in X$. If $h^{-1}(A)\not\ll h^{-1}(B)$ then there is an $\epsilon>0$ such that for all $n\in\mathbb{N}$ there exists $x_{n}\in h^{-1}(A)\setminus \mathscr{B}(x_{0},n)$ and $y_{n}\in(X\setminus h^{-1}(B))\setminus \mathscr{B}(x_{0},n)$ such that $d(x_{n},y_{n})<\epsilon$. The sets $A^{\prime}:=\{x_{n}\}_{n\in\mathbb{N}}$ and $B^{\prime}:=\{y_{n}\}_{n\in\mathbb{N}}$ are unbounded sets such that $A^{\prime}{\bf b}_{d_1} B^{\prime},$ which by the coarseness of $h$ implies that $h(A^{\prime}){\bf b}_{d_2}h(B^{\prime}).$ Therefore, by Lemma \ref{lemma1}, $A{\bf b}_{d_2}(Y\setminus B)$, a contradiction. Thus, $h^{-1}(A)\ll h^{-1}(B)$.
\end{proof}

\begin{remark}
Notice that if $A,B \subseteq X,$ then $A \ll B,$ does not imply $h(A) \ll h(B).$ To see that, let $X= \mathbb{R}, Y=\mathbb{R}^2, A=B=X,$ and let $f: X \to Y$ be defined by $f(x)=(x,0).$ Then $f$ is a coarse map, $A \ll B,$ but $h(A)\not\ll h(B).$
\end{remark}

\section{Equivalence Relation Induced by Coarse Proximity}\label{equivalencerelation}
In this section, we introduce certain equivalence relations on the power set of a space, called weak asymptotic resemblances. We show that every coarse proximity space induces a weak asymptotic resemblance, and consequently every coarse proximity structure induces a coarse structure. We also show that in the case of metric coarse proximity spaces, the weak asymptotic resemblance coincides with the asymptotic resemblance introduced in \cite{Honari}.

Recall the following definition and two examples from \cite{Honari}:
\begin{definition}\label{asymptoticresemblancedefinition}
Let $X$ be a set. Let $\lambda$ be an equivalence relation on the power set of $X.$ Then $\lambda$ is called an \textbf{asymptotic resemblance} if it satisfies the following properties:
\begin{enumerate}[(i)]
\item $A_1\lambda B_1,\,A_2\lambda B_2\text{ implies }(A_1\cup A_2)\lambda(B_1\cup B_2),$
\item $(B_{1}\cup B_{2})\lambda A$ and $B_{1},B_{2}\neq\emptyset$ implies that there are nonempty $A_{1},A_{2}\subseteq A$ such that $A=A_{1}\cup A_{2}$, $B_{1}\lambda A_{1}$, and $B_{2}\lambda A_{2}$.
\end{enumerate}
\end{definition}

\begin{example}
Let $(X,d)$ be a metric space. For any two subsets $A$ and $B$ of $X,$ define $A\lambda_d B$ if $d_{H}(A,B) < \infty.$ Then the relation $\lambda_d$ is an asymptotic resemblance on $X$. We call $\lambda_d$ the \textbf{asymptotic resemblance induced by the metric $d$}.

\end{example}

\begin{example}\label{example6}
Let $\mathcal{E}$ be a coarse structure on a set $X.$ For any two subsets $A$ and $B$ of $X,$ define $A \lambda_{\mathcal{E}}B$ if $A \subset E[B]$ and $B \subset E[A]$ for some $E \in \mathcal{E}.$ Then the relation $\lambda_{\mathcal{E}}$ is an asymptotic resemblance on $X$. We call $\lambda_{\mathcal{E}}$ the \textbf{asymptotic resemblance induced by the coarse structure $\mathcal{E}$}.
\end{example}

\begin{remark}
Without loss of generality we can always assume that the set $E$ from example \ref{example6} is symmetric.
\end{remark}

\begin{definition}\label{weakasymptoticresemblance}
	Let $X$ be a set and $\phi$ an equivalence relation on $2^{X}$ satisfying the following property:	
	\[A\phi B,\,C\phi D\text{ implies }(A\cup C)\phi(B\cup D).\]
	Then we call $\phi$ a \textbf{weak asymptotic resemblance}. If $A\phi B,$ then we say that $A$ and $B$ are \textbf{$\phi$ related}.
\end{definition}

As stated in the following proposition, every weak asymptotic resemblance induces a coarse structure:

\begin{proposition}\label{generatingcoarse}
	Let $X$ be a set and $\phi$ a weak asymptotic resemblance. Then the collection $\mathcal{E}_{\phi}$ of all subsets $E\subseteq X\times X$ such that $\pi_{1}(F)\phi\pi_{2}(F)$ for all $F\subseteq E$ (where $\pi_{1}$ and $\pi_{2}$ denote projection maps onto the first and second factor, respectively) is a coarse structure on $X$.
\end{proposition}
\begin{proof}
	See Proposition 3.2 of \cite{Honari}.
\end{proof}

Every weak asymptotic resemblance induces a coarse structure, and every coarse structure induces an asymptotic resemblance. The following result shows that composition of these two operations does not enlarge the collection of related sets.

\begin{proposition}\label{almost}
	Let $X$ be a set and $\phi$ a weak asymptotic resemblance on $X.$ Let $\mathcal{E}_{\phi}$ be the coarse structure induced by that relation, as in proposition \ref{generatingcoarse}. Then the asymptotic resemblance relation induced by $\mathcal{E}_{\phi}$ is a subset of $\phi$.
\end{proposition}
\begin{proof} Let $A,B\subseteq X$ such that $A\lambda_{\mathcal{E}_{\phi}} B.$ Then there exists a symmetric $E\in\mathcal{E}_{\phi}$ such that $A\subseteq E[B],\,B\subseteq E[A],$ i.e., the following are satisfied:

\begin{enumerate}[(i)]
\item for all $a \in A$, there exists $b \in B$ such that $(a,b) \in E,$
\item for all $b \in B$, there exists $a \in A$ such that $(b,a) \in E.$
\end{enumerate}
Since $E$ is symmetric, these are equivalent to the following:
\begin{enumerate}[(i)]
\item for all $a \in A$, there exists $b \in B$ such that $(a,b) \in E,$
\item for all $b \in B$, there exists $a \in A$ such that $(a,b) \in E.$
\end{enumerate}
Let $F$ be a subset of $E$ that consists of the union of the points $(a,b)$ described in conditions (i) and (ii). Then clearly $\pi_{1}(F)=A$ and $\pi_{2}(F)=B,$ which by the definition of the coarse structure induced by $\phi$ implies that $A\phi B.$
\end{proof}

Now we are going to show that every coarse proximity space induces a weak asymptotic resemblance.

\begin{theorem}\label{asymptoticresemblance}
	Let $(X,\mathcal{B},{\bf b})$ be a coarse proximity space. Let $\phi$ be the relation on the power set of $X$ defined in the following way: $A\phi B$ if and only if the following hold:
	\begin{enumerate}[(i)]
	\item for every unbounded $B^{\prime}\subseteq B$ we have $A{\bf b}B^{\prime},$
	\item for every unbounded $A^{\prime}\subseteq A$ we have $A^{\prime}{\bf b}B.$
	\end{enumerate}	
	Then  $\phi$ is a weak asymptotic resemblance that we call the \textbf{weak asymptotic resemblance induced by the coarse proximity ${\bf b}$.} If the coarse proximity is induced by a metric $d,$ then we call $\phi$ the \textbf{weak asymptotic resemblance induced by $d$}.
\end{theorem}

To prove the above theorem, we need the following remarks and lemmas.

\begin{remark}\label{remark8}
If $\phi$ is the relation defined in Theorem \ref{asymptoticresemblance} and $A$ and $B$ are bounded, then they are always $\phi$ related. If $A$ is bounded and $B$ unbounded, then they are not $\phi$ related.
\end{remark}

\begin{remark}\label{remark9}
If $\phi$ is the relation defined in Theorem \ref{asymptoticresemblance}, then notice that $A \phi A$ for all subsets $A$ of $X.$ Also, for all $A,B \subseteq X$ we have $A \phi B$ if and only if $B \phi A.$
\end{remark}

\begin{remark}
Notice that the $\phi$ relation defined in Theorem \ref{asymptoticresemblance} cannot be an asymptotic resemblance, since we have that the empty set is $\phi$ related to any bounded set, whereas in any asymptotic resemblance space the empty set is only related to itself.
\end{remark}

\begin{lemma}\label{transitivity}
Let $(X,\mathcal{B},{\bf b})$ be a coarse proximity space. Then the relation $\phi$ defined in Theorem \ref{asymptoticresemblance} is transitive.
\end{lemma}

\begin{proof}
Let $A, B,$ and $C$ be subsets of $X$ such that $A \phi B$ and $B \phi C.$ Then either all of them are bounded or all of them are unbounded. If all of them are bounded, then by remark \ref{remark8} we have $A \phi C.$ So let us assume that all of them are unbounded. For contradiction, assume $A \bar{\phi} C$. Then, without loss of generality there exists an unbounded set $A^{\prime} \subseteq A$ such that $A^{\prime}\bar{{\bf b}}C$ (the other case will follow similarly by symmetry). Thus, there exists an unbounded set $E$ such that $A^{\prime}\bar{\bf b}E$ and $(X\setminus E)\bar{\bf b}C.$ If there exists an unbounded $B^{\prime} \subseteq B$ such that $B^{\prime} \subseteq (X\setminus E),$ then $(X\setminus E)\bar{\bf b}C$ and remark \ref{remark11} imply that $B^{\prime}\bar{\bf b}C,$ a contradiction to $B \phi C.$ Thus, it has to be that $B \subseteq E$ up to some bounded set $D$, i.e., $(B \setminus D) \subseteq E.$ Thus, since $A^{\prime}\bar{\bf b}E,$  by remark \ref{remark11} we must have that $A^{\prime}\bar{\bf b}(B \setminus D),$ which by proposition \ref{proposition2} implies that $A^{\prime} \bar{\bf b}B,$ a contradiction to $A \phi B.$ Therefore, it has to be that $A{\phi} C.$
\end{proof}

\begin{lemma}\label{lemma9}
Let $(X,\mathcal{B},{\bf b})$ be a coarse proximity space and let $\phi$ be the relation on the power set of $X$ as defined in Theorem \ref{asymptoticresemblance}. If $A\phi B,$ then for any bounded sets $D_1$ and $D_2,$ we have $(A \cup D_1) \phi (B \cup D_2).$
\end{lemma}

\begin{proof}
If $A$ and $B$ are bounded, then the result follows from Remark \ref{remark8}. So let us assume that $A$ and $B$ are unbounded. Let $E \subseteq A\cup D_1$ be unbounded. Then there exists unbounded $E' \subseteq E$ such that $E' \subseteq A.$ Thus, since $A\phi B,$ we have $E'{\bf b}B,$ which by Lemma \ref{lemma1} implies that $E{\bf b}(B \cup D_2.)$ The other condition follows similarly.
\end{proof}

Finally we are ready to prove Theorem \ref{asymptoticresemblance}.

\begin{proof}[Proof of Theorem \ref{asymptoticresemblance}]
The fact that $\phi$ is an equivalence relation follows from Remark \ref{remark9} and Lemma \ref{transitivity}. To see that $\phi$ satisfies the property from Definition \ref{weakasymptoticresemblance}, let $A,B,C,D\subseteq X$ be such that $A\phi B$ and $C\phi D$. If either pair ($A$ and $B$ or $C$ and $D$) is bounded, then the result follows from lemma \ref{lemma9}. Therefore, we will assume that all of them are unbounded. Now let $E \subseteq A\cup C$ be an unbounded set. Then either $E\cap A$ or $E\cap C$ is unbounded. Let us call that unbounded set $E'$. Then we have either $E' {\bf b}B$ or $E' {\bf b}D$, which by Lemma \ref{lemma1} implies that $E {\bf b} (B \cup D).$ Similarly in the reverse direction. Thus $(A\cup C)\phi(B\cup D)$. 
\end{proof}

The following corollary shows that coarse proximities induce coarse structures.
\begin{corollary}
Let $(X,\mathcal{B},{\bf b})$ be a coarse proximity space and let $\phi$ be the weak asymptotic resemblance induced by coarse proximity ${\bf b}.$ Then $\phi$ induces a coarse structure on $X.$
\end{corollary}

\begin{proof}
This is an immediate consequence of Proposition \ref{generatingcoarse} and Theorem \ref{asymptoticresemblance}.
\end{proof}

Next proposition implies that in the case of metric spaces, the induced weak asymptotic resemblance and the induced asymptotic resemblance coincide when one considers nonempty subsets.
\begin{proposition}\label{metricconsistency}
	Let $(X,d)$ be a metric space and let $\phi$ be the weak asymptotic resemblance induced by the metric $d.$ Then given nonempty $A,B\subseteq X,$ we have that $A\phi B$ if and only if $A$ and $B$ have finite Hausdorff distance.
\end{proposition}

\begin{proof}
To prove the forward direction, assume that $A\phi B$ and assume towards a contradiction that $d_{H}(A,B)=\infty$. Then for each $n\in\mathbb{N}$ there exists $x_{n}\in A$ such that $d(x_{n},B)>n$ or there exists $y_n \in B$ such that $d(y_{n},A)>n$. Define $A^{\prime}$ to be the collection of all such $x_n$ and $B^{\prime}$ to be the collection of all such $y_n.$ Without loss of generality we may assume that $A^{\prime}$ is not finite. Notice that $A^{\prime}$ has to be unbounded (if $A^{\prime}$ is bounded, then $d(a_i, a_j)<M$ for all $a_i, a_j \in A^{\prime}$. Let $a_k \in A^{\prime}$. Then $d(a_k, B) \leq N$ for some $N,$ and consequently $d(a_i, B) \leq M+N$ for all $a_i \in A^{\prime},$ a contradiction to the construction of $A'$). Because $A\phi B$ we have that $A^{\prime}{\bf b}B,$ which implies that there are unbounded subsets $A^{\prime\prime}\subseteq A^{\prime}$ and $B^{\prime\prime}\subseteq B$ such that $d_{H}(A^{\prime\prime},B^{\prime\prime}) \leq n$ for some $n < \infty.$ Therefore, for all $a \in A^{\prime\prime}$ there exists $b \in B^{\prime\prime}$ such that $d(a,b)<n,$ a contradiction to the construction of $A^{\prime}.$

To prove the converse direction, assume that $d_{H}(A,B)=m < \infty$. If $A$ and $B$ are bounded, then $A\phi B$ trivially. If $A$ and $B$ are unbounded and $A^{\prime}\subseteq A$ is an unbounded set, then we know that $A^{\prime}\subseteq \mathscr{B}(B,m).$ Therefore, for all $a \in A^{\prime}$ we can find $b_a \in B$ such that $d(a,b_a) <m$. Let $B^{\prime}=\{b_{a}\}_{a\in A^{\prime}}.$ Then by construction of $B^{\prime}$ we have that $d_{H}(A^{\prime},B^{\prime})\leq m< \infty.$ which implies that $A^{\prime}{\bf b}B$. Running through the same argument replacing $A$'s with $B$'s yields $A\phi B$.
\end{proof}

The above proposition also implies that in the case of a metric space $(X,d),$ the underlying coarse proximity relation induces the asymptotic resemblance induced by $d$ when one considers nonempty subsets.

The following proposition shows that in any coarse proximity space two subsets are $\phi$ related if and only if they share all coarse neighborhoods.

\begin{proposition}\label{coarseneighborhoodinclusion}
	Let $(X,\mathcal{B},{\bf b})$ be a coarse proximity space and $\phi$ the weak asymptotic resemblance induced by the coarse proximity ${\bf b}.$ Then for $A,B\subseteq X$ the following are equivalent:
	\begin{enumerate}[(i)]
		\item For all $C\subseteq X,$ $A\ll C$ if and only if $B\ll C,$
		\item $A \phi B.$
	\end{enumerate}
\end{proposition}
\begin{proof}
 ($(ii) \implies (i)$) Assume $A \phi B$ and let $C$ be such that $A\ll C.$ Proposition \ref{intermediateneighborhood} implies the existence of $E$ such that $A\ll E \ll C.$ Notice that $B \subseteq E$ up to a bounded set $D,$ i.e., $(B \setminus D) \subseteq E.$ For if that is not the case, then $D$ is an unbounded subset of $ X \setminus E$ such that $D {\bf b} A$ (because $D \subseteq B$ and $A \phi B$), and therefore implying that $(X \setminus E) {\bf b} A,$ a contradiction to $A\ll E.$ Thus, we know that $(B \setminus D) \subseteq E$ and since $E\ll C$ ($i.e., E \bar{\bf b} (X \setminus C)$), we have that $(B \setminus D) \bar{\bf b} (X \setminus C)$, which by Proposition \ref{proposition2} shows that $B \bar{\bf b} (X \setminus C),$ i.e. $B\ll C.$ The other implication follows by symmetry.

($(i)\implies (ii)$) Let $B^{\prime}\subseteq B$ be an unbounded subset and assume towards a contradiction that $A\bar{{\bf b}}B^{\prime}$. Then by the strong axiom there is an $E\subseteq X$ such that $A\ll (X\setminus E)$ and $B^{\prime}\ll E.$ However, by assumption we have that $B\ll (X\setminus E).$ In particular, this implies that $B^{\prime}\ll (X\setminus E).$ So we have that $B^{\prime}\ll E$ and $B^{\prime}\ll (X\setminus E),$ which by Proposition \ref{propertiesofcoarseneighborhoods} implies that $B^{\prime}$ is bounded, a contradiction. Therefore $A{\bf b}B^{\prime}$ for every unbounded $B^{\prime}\subseteq B$. Similarly one can show that $A^{\prime} {\bf b}B$ for every unbounded $A^{\prime}\subseteq A$. Thus $A \phi B.$
\end{proof}

One could expect that $A \phi B$ implies that for all $C\subseteq X$ $C\ll A$ if and only if $C\ll B.$ However, that is not the case.

\begin{example}
Consider $\mathbb{R}^2.$ Let $A=\{(x,y) \mid y=\abs{x} \},$ $B=\{(x+1,y) \mid (x,y) \in A \},$ and $C=\{(x,y) \mid (x,y) \in A \text{ and } x \leq 0 \}.$ Let $X=A \cup B$ with the metric inherited from $\mathbb{R}^2.$ Then $A \phi B$ and $C\ll A,$ but it is not true that $C\ll B$ (in fact, $C$ is unbounded and disjoint from $B$).
\end{example}

The following corollary is a direct consequence of Proposition \ref{coarseneighborhoodinclusion} and will be used in section \ref{coarsecategory}.

\begin{corollary}\label{corollaryforsection7}
Let $(X,\mathcal{B},{\bf b})$ be a coarse proximity space and $\phi$ the corresponding equivalence relation on $2^{X}$. Let $A,B,C,$ and $D$ be subsets of $X$ such that $A \phi C$ and $B \phi D$. Then $A {\bf b} B$ if and only if $C {\bf b} D.$
\end{corollary}
\begin{proof}
Assume $A \bar{\bf b}B.$ Then there exists $E \subseteq X$ such that $E \bar{\bf b} A$ and $(X \setminus E) \bar{\bf b} B.$ By Remark \ref{strongaxiomneighborhoods} this can be translated to $A\ll (X \setminus E)$ and $B\ll E.$ By Proposition \ref{coarseneighborhoodinclusion}, this implies that $C\ll (X \setminus E)$ and $D\ll E,$ i.e. $E \bar{\bf b} C$ and $(X \setminus E) \bar{\bf b} D.$ By the converse of the strong axiom, this implies that $C \bar{\bf b}D.$ The converse direction follows by symmetry.
\end{proof}

\section{Coarse proximity maps}\label{coarsecategory}

In this section, we introduce functions preserving coarse proximity relations, called coarse proximity maps, and we investigate their basic properties. We also show that the collections of coarse proximity spaces and closeness classes of coarse proximity maps make up a category.

\begin{definition}
	Let $(X,\mathcal{B}_{1},{\bf b}_{1})$ and $(Y,\mathcal{B}_{2},{\bf b}_{2})$ be coarse proximity spaces. Let $f:X\rightarrow Y$ be a function and $A$ and $B$ subsets of $X.$ Then $f$ is a \textbf{coarse proximity map} provided that the following are satisfied:
\begin{enumerate}[(i)]
\item $B\in\mathcal{B}_{1}$ implies $f(B)\in\mathcal{B}_{2},$
\item $A{\bf b}_{1}B$ implies $f(A){\bf b}_{2}f(B).$
\end{enumerate}
\end{definition}

\begin{remark}\label{remark5}
Notice that a coarse proximity map sends unbounded sets to unbounded sets. For if $B \notin \mathcal{B}_1,$ then $B {\bf b}_1B.$ Thus, $f(B) {\bf b}_2 f(B),$ implying that $f(B) \notin \mathcal{B}_2.$ Consequently, preimages of bounded sets are bounded.
\end{remark}

\begin{remark}\label{remark13}
Notice that the composition of two coarse proximity maps is a coarse proximity map.
\end{remark}

The following proposition shows that in the case of metric spaces, coarse maps and coarse proximity maps coincide.

\begin{proposition}\label{coarseequalsproximity}
	Let $(X,\mathcal{B}_{d_1},{\bf b}_1)$ and $(Y,\mathcal{B}_{d_2},{\bf b}_2)$ be metric coarse proximity spaces. Let $f:X\rightarrow Y$ be a function. Then $f$ is a coarse map if and only if $f$ is a coarse proximity map.
\end{proposition}
\begin{proof}
To prove the forward direction, assume that $f$ is a coarse map. Since $f$ is bornologous, it sends bounded sets to bounded sets. Now assume $A,B \subseteq X$ are such that $A{\bf b}_1 B,$ and for contradiction assume that $f(A) \bar{\bf b}_2 f(B).$ Then there exists a set $E\subseteq Y$ such that  $f(A) \bar{\bf b}_2 E$ and  $(Y\setminus E) \bar{\bf b}_2 f(B),$ i.e.,
\[f(A)\ll (Y \setminus E) \quad \text{and} \quad f(B) \ll E.\]
Since coarse maps copreserve coarse neighborhoods (see Proposition \ref{copreservingneighborhoods}), this implies that
\[A \subseteq f^{-1}(f(A)) \ll f^{-1} (Y \setminus E)=(X \setminus f^{-1}(E))  \quad \text{and} \quad B \subseteq f^{-1}(f(B))\ll f^{-1}(E),\]
i.e., $A \bar{\bf b} f^{-1}(E)$ and $B \bar{\bf b} (X\setminus f^{-1}(E)).$ By Proposition \ref{converse_of_strong_axiom}, this shows that $A\bar{\bf b} B,$ a contradiction. Thus, it has to be that $f(A){\bf b} f(B),$ completing the proof that $f$ is a coarse proximity map.

To prove the converse, let $f$ be a coarse proximity and let $\lambda_{1}$ and $\lambda_{2}$ be asymptotic resemblance relations induced by the metrics $d_{1}$ and $d_{2}$ respectively. By Proposition \ref{metricconsistency} we have that for nonempty subsets these relations are precisely the $\phi_1$ and $\phi_2$ relations constructed from the respective coarse proximity structures as in Proposition \ref{asymptoticresemblance} (i.e., they are weak asymptotic resemblances induced by $d_1$ and $d_2,$ respectively). We will show that $f$ is an asymptotic resemblance map, as in \cite{Honari}. Let $A,B\subseteq X$ be such that $A\lambda_{1}B$. It is trivial to show that $f(A)\lambda_2 f(B)$ (the implication $A \phi B \implies f(A) \phi f(B)$ is actually true for any coarse proximity map. For the proof, see Proposition \ref{proximitymapspreserveweakasymptoticresemblance}). Thus, $f$ is an asymptotic resemblance mapping as in \cite{Honari}. Since Remark \ref{remark5} implies that $f$ is also proper, by Theorem $2.3$ of \cite{Honari} $f$ must also be a coarse mapping between the metric spaces $(X,d_{1})$ and $(Y,d_{2})$. 
\end{proof}

\begin{corollary}\label{allthesame}
Let $X$ and $Y$ be metric spaces and let $f:X \to Y$ be a function. Then $f$ is a coarse map if and only if $f$ is a coarse proximity map if and only if $f$ is an asymptotic resemblance map.
\end{corollary}
\begin{proof}
This follows from Proposition \ref{coarseequalsproximity} and Proposition $2.3$ of \cite{Honari}.
\end{proof}

The following corollary shows that if $X$ is a metric space, then any coarse proximity map copreserves coarse neighborhoods.
\begin{corollary}
	Let $(X,d_{1}),(Y,d_{2})$ be metric spaces and $h:X\rightarrow Y$ a coarse proximity map. If $A,B\subseteq Y$ such that $A\ll B$ with respect to the metric coarse proximity structure induced by $d_2,$ then $h^{-1}(A)\ll h^{-1}(B)$ with respect to the metric coarse proximity structure on $X$ induced by $d_1.$
\end{corollary}

\begin{proof}
This is an immediate consequence of Proposition \ref{coarseequalsproximity} and Proposition \ref{copreservingneighborhoods}.
\end{proof}

As is usual for coarse topology, the morphisms in the category of coarse proximity spaces will not simply be coarse proximity maps, but instead equivalence classes thereof. We take our definition of closeness to be aesthetically similar to the definition of closeness for maps whose codomain is an asymptotic resemblance space, as in \cite{Honari}.

\begin{definition}\label{coarsecloseness}
	Let $X$ be a set and $(Y,\mathcal{B},{\bf b})$ a coarse proximity space. Two functions $f,g:X \to Y$ are {\bf close}, denoted $f\sim g$, if for all $A \subseteq X$
	\[f(A)\phi g(A),\]
	where $\phi$ is the weak asymptotic resemblance relation induced by the coarse proximity structure ${\bf b}.$
\end{definition}

\begin{remark}
Notice that since $\phi$ is an equivalence relation, the closeness relation from Definition \ref{coarsecloseness} is an equivalence relation. We will denote the equivalence class of a function $f$ by $[f].$
\end{remark}

There are at least $3$ ways to define closeness relation on maps from $X$ to $Y.$ If $Y$ is a coarse proximity space, we can define the closeness relation with respect to that relation, as in Definition \ref{coarsecloseness}. If $Y$ is a metric space, then we can define the closeness relation with respect to that binary operation, as in Definition \ref{coarselyclose}. Finally, if $Y$ is an asymptotic resemblance space, we can define the closeness relation with respect to that relation, as in \cite{Honari}. The following proposition shows that in the case of metric spaces, all of these definitions of closeness coincide. To easily distinguish between the closeness relations, for the remainder of this section we will say that $f$ and $g$ are proximally close if they satisfy Definition \ref{coarsecloseness}, coarsely close if they satisfy Definition \ref{coarselyclose}, and asymptotically close if they satisfy Definition $2.15$ of \cite{Honari}.

\begin{proposition}\label{equivalenceofcloseness}
	Let $X$ be a set, $(Y,d)$ a metric space, and $f,g:X \to Y$ two functions. Then the following are equivalent:
	\begin{enumerate}[(i)]
	\item $f$ and $g$ are proximally close, \label{equivalent1}
	\item $f$ and $g$ are asymptotically close, \label{equivalent2}
	\item $f$ and $g$ are coarsely close. \label{equivalent3}
	\end{enumerate}
\end{proposition}

\begin{proof}
Since in the case of metric spaces asymptotic resemblance induced by the metric and the $\phi$ relation coincide for nonempty sets (see Proposition \ref{metricconsistency}), the closeness relation from Definition \ref{coarsecloseness} (i.e., the definition of proximally close) coincides with the closeness relation defined in \cite{Honari} (i.e., the definition of asymptotically close). This shows the equivalence of (\ref{equivalent1}) and (\ref{equivalent2}). The equivalence of (\ref{equivalent2}) and (\ref{equivalent3}) is the statement of Proposition 2.16 in \cite{Honari}.
\end{proof}

\begin{remark}
Thanks to the above proposition, whenever we deal with metric spaces, the sentence "closeness class of a function $f$" is unambiguous.
\end{remark}

\begin{corollary}\label{extraequivalenceofcloseness}
	Let $f,g:(X,d_{1})\rightarrow(Y,d_{2})$ be maps between metric spaces. Then the following are equivalent:
	\begin{enumerate}[(i)]
	\item $f$ and $g$ are coarse proximity maps and are proximally close,
	\item $f$ and $g$ are asymptotic resemblance maps and are asymptotically close,
	\item $f$ and $g$ are coarse maps and are coarsely close.
	\end{enumerate}
\end{corollary}

\begin{proof}
This follows immediately from Proposition \ref{allthesame} and Proposition \ref{equivalenceofcloseness}.
\end{proof}

\begin{definition}
Let $(X,\mathcal{B}_{1},{\bf b}_{1})$ and $(Y,\mathcal{B}_{2},{\bf b}_{2})$ be coarse proximity spaces. We call a coarse proximity map $f: X \to Y$ a \textbf{proximal coarse equivalence} if there exists a coarse proximity map $g:Y\to X$ such that $g\circ f\sim id_{X}$ and $f\circ g\sim id_{Y}.$ We say that $(X,\mathcal{B}_{1},{\bf b}_{1})$ and $(Y,\mathcal{B}_{2},{\bf b}_{2})$ are \textbf{proximally coarse equivalent} if there exists a proximal coarse equivalence $f:X \to Y.$
\end{definition}

\begin{remark}
Notice that Proposition \ref{allthesame} and Proposition \ref{equivalenceofcloseness} also imply that the proximal coarse equivalence coincides with asymptotic equivalence (defined in \cite{Honari}) and coarse equivalence (defined in \cite{Roe}).
\end{remark}

To define a reasonable definition of composition of two closeness classes of coarse proximity maps, we need to know that coarse proximity functions preserve the $\phi$ relation.

\begin{proposition}\label{proximitymapspreserveweakasymptoticresemblance}
Let $(X,\mathcal{B}_{1},{\bf b}_{1})$ and $(Y,\mathcal{B}_{2},{\bf b}_{2})$ be coarse proximity spaces and let $f:X \to Y$ be a coarse proximity map. Let $\phi_1$ and $\phi_2$ be weak asymptotic resemblance relations induced by ${\bf b}_1$ and ${\bf b}_2$, respectively. Then for any $A,B \subseteq X,$ we have
\[A \phi_1 B \implies f(A) \phi_2 f(B).\]
\end{proposition}

\begin{proof}
Let $A,B,$ and $f$ be as in the statement of the proposition. If $A$ and $B$ are bounded, then the result is trivial. So assume that $A$ and $B$ are unbounded. For contradiction assume that $f(A) \bar{\phi_2}f(B).$ Then there exists $A^{\prime} \subseteq f(A)$ such that $A^{\prime}$ is unbounded and $A' \bar{{\bf b}_2}f(B)$. Then $A'':=f^{-1}(A') \cap A$ is unbounded, $A'' \subseteq A$ and $A'' \bar{{\bf b}_1} B$ (because otherwise $f(A''){\bf b}_2 f(B),$ and since $f(A'') \subseteq f(A),$ $f(A){\bf b}_2 f(B)$), a contradiction to $A \phi_1 B.$
\end{proof}

The following proposition implies that if $f \sim g$, then $f$ is a coarse proximity map/equivalence if and only if $g$ is.

\begin{proposition}
Let $(X,\mathcal{B}_{1},{\bf b}_{1})$ and $(Y,\mathcal{B}_{2},{\bf b}_{2})$ be coarse proximity spaces. Let $f:X \to Y$ and $g:X \to Y$ be two close functions. If $f$ is a coarse proximity map, then so is $g$. If $f$ is a proximal coarse equivalence, then so is $g.$
\end{proposition}

\begin{proof}
Let $\phi_1$ and $\phi_2$ be weak asymptotic resemblance relations induced by ${\bf b}_1$ and ${\bf b}_2$, respectively. Let us first assume that $f$ is a coarse proximity map. Let $B \subseteq X$ be bounded. Since $f$ is a coarse proximity map, $f(B)$ is bounded. Since $f(B) \phi_2 g(B),$ Remark \ref{remark8} implies that $g(B)$ is bounded. Now let $A, C \subseteq X$ and assume $A{\bf b}_1C.$ Since $f$ is a coarse proximity map, $f(A){\bf b}_2f(C).$ Since $g$ is close to $f$, we have that $f(A) \phi_2 g(A)$ and $f(C) \phi_2 g(C).$ Then Corollary \ref{corollaryforsection7} implies that $g(A){\bf b}_2g(C).$ Thus, $g$ is a coarse proximity map.

Now assume that $f$ is a proximal coarse equivalence, i.e., there exists a coarse proximity map $f_1:Y\to X$ such that $f_1\circ f\sim id_{X}$ and $f\circ f_1\sim id_{Y}.$ We will show that $f_1\circ g\sim id_{X}$ and $g\circ f_1\sim id_{Y}.$ To see that $f_1\circ g\sim id_{X},$ let $A \subseteq X.$ Then since $g \sim f,$ we have that $g(A) \phi_2 f(A).$ Since $f_1$ is a coarse proximity map, Proposition \ref{proximitymapspreserveweakasymptoticresemblance} implies that $(f_1(g(A)) \phi_1 (f_1(f(A)).$ Since $A$ was arbitrary, this implies that
\[(f_1\circ g) \sim (f_1\circ f) \sim id_X.\]
To see that $g \circ f_1\sim id_{Y}$, let $C\subseteq Y.$ Since $g$ is close to $f,$ we have $g(f_1(C)) \phi_2 f(f_1(C)).$ Since $C$ was arbitrary, this implies that
\[(g\circ f_1) \sim (f_1\circ f_1) \sim id_Y. \qedhere\]
\end{proof}

\begin{proposition}\label{compositionofcoarsnessclasses}
Let $(X,\mathcal{B}_{1},{\bf b}_{1}), (Y,\mathcal{B}_{2},{\bf b}_{2}),$ and $(Z,\mathcal{B}_{3},{\bf b}_{3})$ be coarse proximity spaces. Let $f: X \to Y$ and $g:Y \to Z$ be coarse proximity functions and let $[f]$ and $[g]$ be respective closeness classes. Then the operation $[f] \circ [g]:= [f \circ g]$ is well-defined.
\end{proposition}

\begin{proof}
Let $f$ and $g$ be as in the statement of the proposition. Let $\phi_2$ and $\phi_3$ be weak asymptotic resemblance relations induced by ${\bf b}_2$ and ${\bf b}_3,$ respectively. Let $f' \in [f]$ and $g' \in [g].$ By Remark \ref{remark13}, $g \circ f$ and $g' \circ f'$ are coarse proximity maps from $X$ to $Z.$ Let us show that $g \circ f$ and $g' \circ f'$ are close, which will show that $[g \circ f]=[g' \circ f'].$
Let $A$ be a set. Since $f \sim f'$, we have that $f(A) \phi_2 f'(A).$ Therefore, we have
\[\Big(g(f(A))\Big) \phi_3 \Big(g(f'(A))\Big) \phi_3 \Big(g'(f'(A))\Big),\]
where the first equivalence follows from Proposition \ref{proximitymapspreserveweakasymptoticresemblance} and the second equivalence follows from $g \sim g'.$ Since $\phi_3$ is an equivalence relation, this completes the proof that $[g \circ f]=[g' \circ f'].$
\end{proof}

\begin{definition}\label{categoryofcoarseproximityspaces}
The collection of coarse proximity spaces and closeness classes of coarse proximity maps (with the composition of morphisms defined as in Proposition \ref{compositionofcoarsnessclasses}) makes up the category $C_{\bf b}$ of coarse proximity spaces.
\end{definition}

\begin{remark}
Associativity of morphisms in the above definition follows from the associativity of composition of functions. The identity morphism is the equivalence class of the identity map.
\end{remark}

\begin{remark}
Notice that if $(X,\mathcal{B}_{1},{\bf b}_{1})$ and $(Y,\mathcal{B}_{2},{\bf b}_{2})$ are coarse proximity spaces and $f:X \to Y$ is a proximal coarse equivalence, then $[f]$ is an equivalence in the category of coarse proximity spaces.
\end{remark}

\section{Proximity at Infinity}\label{hyperspace}

In this section, we construct a natural small-scale proximity structure on the set of unbounded subsets of a metric space. We also show how this structure naturally induces a small-scale proximity on the equivalence classes of the weak asymptotic resemblance induced by the metric. We call this space the proximity space at infinity. We then proceed to show that the construction is functorial, making up a functor from the category of unbounded metric spaces whose morphisms are closeness classes of coarse proximity maps (or coarse maps) to the category of proximity spaces whose morphisms are proximity maps. The idea of defining topological structures on equivalence classes of unbounded sets has been utilized previously. In \cite{Hartmann}, a functor from metric spaces to totally bounded metric spaces, called "spaces of ends", is constructed. For a variety of unbounded metric spaces the space of ends is empty. As we will see, the proximity space at infinity for every unbounded metric space is always nonempty. Our construction was inspired by considering the Vietoris topology on the hyperspace of the Higson corona of a proper metric space.

\begin{definition}
	A sequence $f:\mathbb{N}\rightarrow\mathbb{R}$ is called \textbf{adequate} if it is positive and $f(n)-f(n-1)>n+1$ for all $n>1.$
\end{definition}

\begin{remark}
Notice that if $f$ and $g$ are adequate sequences, then so is $h:=\max{\{f,g\}}.$
\end{remark}

\begin{definition}\label{importantdefinition}
	Let $(X,d)$ be a metric space and let $x_{0}$ be a point in $X.$ If $A\subseteq X$ is a set and $f:\mathbb{N}\rightarrow\mathbb{R}$ is an adequate sequence, then we define the \textbf{coarse neighborhood of $A$ of radius $f$ relative to $x_{0}$}, denoted $U_{x_{0}}(A,f)\subseteq X,$ in the following way:	
\begin{equation*}
\begin{split}
A_{0}^{f} & :=A,\\
A_{n}^{f} & := \mathscr{B}(A,n)\setminus \mathscr{B}(x_{0},f(n)),\\
U_{x_{0}}(A,f) & :=\bigcup_{n \geq 0} A_{n}^{f}.
\end{split}
\end{equation*}

To simplify notation, when the base point is clear from the context we will denote $\mathscr{B}(x_{0},f(n))$ by $\mathscr{B}_{f(n)}$ and $U_{x_{0}}(A,f)$ by $U(A,f).$ In this notation, the definition of $U(A,f)$ becomes
\begin{equation*}
\begin{split}
A_{0}^{f} & :=A,\\
A_{n}^{f} & := \mathscr{B}(A,n)\setminus \mathscr{B}_{f(n)},\\
U(A,f) & :=\bigcup_{n \geq 0} A_{n}^{f}.
\end{split}
\end{equation*}
\end{definition}
 
 The reader is encouraged to compare the above definition to the construction of the coarse neighborhood in section \ref{mainbody}. As expected, we will show that a coarse neighborhood of $A$ of radius $f$ relative to $x_{0}$ is really a coarse neighborhood.

\begin{proposition}
	Given a metric space $(X,d)$, a point $x_{0} \in X$, and a set $A\subseteq X,$ we have $A\ll U(A,f)$ for every adequate sequence $f$.
\end{proposition}

\begin{proof}
For contradiction assume that $A{\bf b}(X\setminus U(A,f)).$ Then there exists $\epsilon < \infty$ such that for all $n \in \mathbb{N},$ there exists $a_n \in A \setminus \mathscr{B}_{f(n)}$ and $x_n \in (X \setminus U(A,f)) \setminus \mathscr{B}_{f(n)}$ such that $d(a_n,x_n)<\epsilon.$ Choose $n$ large such that $\epsilon < n.$ Then $x_n \notin \mathscr{B}_{f(n)}$ and $d(a_n, x_n) < \epsilon < n.$ Thus, $x_n \in A_{n}^{f}$, contradicting the fact that $x_n \notin U(A,f).$ Therefore, by contradiction, $A\bar{\bf b}(X\setminus U(A,f)),$ i.e., $A\ll U(A,f).$
\end{proof}

The following definition and proposition justify why it is reasonable to restrict ourselves to considering only coarse neighborhoods of the form $U(A,f).$

\begin{definition}
Given a coarse proximity space $(X,\mathcal{B},{\bf b})$ and a set $A\subseteq X,$ we say that a collection $\mathcal{A}\subseteq 2^{X}$ of coarse neighborhoods of $A$ is a {\bf coarse neighborhood base} at $A$ if for every coarse neighborhood $D\subseteq X$ of $A$ there is $E\in\mathcal{A}$ such that 
	$$A\ll E \ll D.$$
\end{definition}

The following proposition shows that for any set $B,$ all $U(B,f)$ form a coarse neighborhood base at $B.$

\begin{proposition}\label{metricbase}
	Let $(X,d)$ be a metric space and $x_{0}\in X$ a point. For each set $B\subseteq X$ define $\mathcal{C}_{x_{0}}(B)$ to be the set of all coarse neighborhoods of the form $U(B,f),$ where $f$ is an adequate sequence. Then $\mathcal{C}_{x_{0}}(B)$ is a coarse neighborhood base at $B$. 
\end{proposition}
\begin{proof}
The statement is trivial if $B$ is bounded, so assume that $B$ is unbounded. Let $B\subseteq X$ be an unbounded set and $D\subseteq X$ a coarse neighborhood of $B$. Then $B\bar{\bf b}(X\setminus D)$. Set $A=(X\setminus D).$ Then the set $E$ from Theorem \ref{theorem1} is the desired coarse neighborhood such that $E \in \mathcal{C}_{x_{0}}(B)$ and $B \ll E \ll D.$
\end{proof}

Let us explore a few basic properties of coarse neighborhoods of the form $U(A,f).$
\begin{proposition}\label{basicproperties}
Let $(X,d)$ be a metric space, $x_{0}\in X$ a point, $f$ and $g$ adequate sequences, and $A$ and $B$ unbounded subsets of $X.$ Then the following are true:
\begin{enumerate}[(i)]
\item $A \subseteq U(A,f),$
\item if $A$ is bounded, then so is $U(A,f),$
\item if $B \subseteq A,$ then $U(B,f) \subseteq U(A,f),$
\item if $f \leq g,$ then $U(A,g) \subseteq U(A,f),$
\item $U(A,f) \cup U(B,f) = U(A \cup B,f),$
\item if $A\ll B,$ then there exists a bounded set $D$ such that $U(A,f) \setminus D \subseteq U(B,f).$
\end{enumerate}
\end{proposition}

\begin{proof}
The first four properties are direct consequences of definitions. (v) follows from the fact that $\mathscr{B}(A,n) \cup \mathscr{B}(B,n) = \mathscr{B}(A \cup B, n).$ To show (vi), let $A\ll B.$ Then there exists a bounded set $D'$ such that $(A \setminus D') \subseteq B.$ Thus, by (iii), we have $U((A \setminus D'), f) \subseteq U(B,f).$ Thus, by (v), $U(A,f) \subseteq U(B,f) \cup U(D',f).$ Let $D=U(D',f).$ Then by (ii) $D$ is bounded, and we get that $U(A,f) \setminus D \subseteq U(B,f).$
\end{proof}

One could expect that if $A,B$ are subsets of a metric space such that $A\ll B,$ then for every adequate sequence $f$ one has $U(A,f)\ll U(B,f)$. However, this is not the case.

\begin{example}
Consider $\mathbb{R}^2.$ Let $A=B=\{(0,y) \mid y>0\}.$ Let $x_0$ be the origin and let $f$ and $g$ be two adequate sequences such that $f(n)<g(n)$ for all $n \in \mathbb{N}.$ Let $X= (\mathbb{R}^2 \setminus U(A,g)) \cup A.$ Then $A\ll B$ (Since $(X \setminus B)=(X \setminus U(A,g))$), but  it is not true that $U(A,f)\ll U(B,f)$ (since $U(A,f) = U(B,f)$ and $U(A,f)$ is unbounded in $X \setminus U(A,g)$).
\end{example}
 
 One could also expect that if $A,B$ are subsets of a metric space such that $A \phi B,$ then for every adequate sequence $f$ one has:
 \begin{enumerate}[(i)]
\item $U(A,f) \phi U(B,f),$
\item $C\ll U(A,f)$ if and only if $C\ll U(B,f)$ for any $C \subseteq X.$
\end{enumerate}
However, the following example shows that neither of these statements is true.
\begin{example}
	Let $1>\epsilon>0$, $A=\{(0,t)\in\mathbb{R}^{2}\mid t\geq 0\},\,B=\{(-1,t)\in\mathbb{R}^{2}\mid t\geq 0\}$ and $x_{0}=(0,0)$. Notice that $A$ and $B$ have finite Hausdorff distance. Define $f:\mathbb{N}\rightarrow\mathbb{N}$ by $f(n)=n^{3}.$ Then $f$ is an adequate sequence. For each $n\in\mathbb{N},$ define $x_{n}=((n-\epsilon),f(n)+\epsilon)$. Let $C=\{x_{n}\}_{n\in\mathbb{N}}.$ Let $X=A\cup B\cup C$ with the subspace metric inherited from $\mathbb{R}^{2}$. Then $C \subseteq U(A,f)$ and $C \cap U(B,f)= \emptyset.$ Notice that by this specific construction, there are bounded sets $D_{1}$ and $D_{2}$ such that $X=U(A,g)\cup D_{1}$ and $B\cup A=U(B,g)\cup D_{2}$. Consequently, since $X$ and $A\cup B$ do not have finite Hausdorff distance, neither do $U(A,g)$ and $U(B,g)$. Also, notice that $C\ll U(A,f)$ (since $(X \setminus U(A,f)$ is a bounded set), but it is not true that $C\ll U(B,f)$ (since $C \cap (X\setminus U(B,f)) = C,$ which is unbounded).
\end{example}

To be able to prove the "star-refinement" property of coarse neighborhoods, we need the following lemmas:

\begin{lemma}\label{lemmaforstarrefinement}
Let $(X,d)$ be a metric space, $A$ an unbounded subset of $X$, $x_{0}$ a point in $X,$ $f:\mathbb{N}\rightarrow\mathbb{R}$ an adequate sequence, and $n \in \mathbb{N}$ such that $n>1.$ Then 
\[d \Big((X\setminus U(A,f)) \setminus \mathscr{B}_{f(n)}, A \setminus \mathscr{B}_{f(n)}\Big)>n-1.\]
\end{lemma}

\begin{proof}
For contradiction, assume that there exists $x \in (X\setminus U(A,f)) \setminus \mathscr{B}_{f(n)}$ and $a \in A \setminus \mathscr{B}_{f(n)}$ such that $d(x,a) < n.$ Then we have that $x \notin \mathscr{B}_{f(n)}$ and $x \in \mathscr{B}(A,n).$ Thus, $x \in A_n^f,$ a contradiction to $x \notin U(A,f).$
\end{proof}

For the remainder of the paper, we will use the following notation: for each $n,$ define
 \[C_n:= \mathscr{B}_{f(n)} \setminus \mathscr{B}_{f(n-1)}.\]

\begin{lemma}\label{furtherandfurther}
Let $(X,d)$ be a metric space, $A$ an unbounded subset of $X$, $x_{0}$ a point in $X,$ $f:\mathbb{N}\rightarrow\mathbb{R}$ an adequate sequence, and $n \in \mathbb{N}$ such that $n>1.$ If $x \in C_n$ and $y \in X$ such that $d(x,y)<n,$ then $y$ can only belong to $C_{n-1}, C_{n},$ or $C_{n+1}.$ In particular, $y \in \mathscr{B}_{f(n+1)}$ and $y \notin \mathscr{B}_{f(n-2)}.$
\end{lemma}

\begin{proof}
The fact that $y \notin C_k$ for $k \leq n-2$ follows from the fact that $\mathscr{B}_{f(k)} \subseteq \mathscr{B}_{f(n-2)}$ for all $k \leq n-2$ and the fact that for $n>1$ the difference in radii between $\mathscr{B}_{f(n-2)}$ and $\mathscr{B}_{f(n-1)}$ is bigger than $n.$ The fact that $y \notin C_k$ for $k \geq n+2$ follows from the fact that $\mathscr{B}_{f(n+2)} \subseteq \mathscr{B}_{f(k)}$ for all $k \geq n+2$ and the fact that the difference in radii between $\mathscr{B}_{f(n+1)}$ and $\mathscr{B}_{f(n)}$ is bigger than $n+2$ for $n>1.$\end{proof}

\begin{lemma}\label{lemmalemma}
Let $(X,d)$ be a metric space, $A$ an unbounded subset of $X$, $x_{0}$ a point in $X,$ $f:\mathbb{N}\rightarrow\mathbb{R}$ an adequate sequence, and $n \in \mathbb{N}$ such that $n>1.$ If $x \in U(A,f)$ and $x \in \mathscr{B}_{f(n)},$ then there exists $a \in A$ such that $d(x,a)<n.$
\end{lemma}

\begin{proof}
Since $x \in U(A,f),$ we know that $x \in A_{m}^{f}$ for some $m$. Thus, there exists $a \in A$ such that $d(x,a)<m.$ Also, since $x \in A_{m}^{f},$ $x \notin \mathscr{B}_{f(m)}.$ Since $x \in \mathscr{B}_{f(n)},$ it has to be that $m<n.$ Thus $d(x,a)<m<n.$
\end{proof}

\begin{proposition}\label{starrefinement}
	Let $(X,d)$ be a metric space and $x_{0}\in X$. Then given an adequate sequence $f:\mathbb{N}\rightarrow\mathbb{R},$ there is another adequate sequence $g$ such that for all unbounded $A\subseteq X$ we have that
\[U(U(A,g),g)\ll U(A,f).\]
\end{proposition}

\begin{proof}
Let $f$ be as in the statement of the proposition. Define 
\[g(n)=f(n^2).\] 
Since $f$ is adequate, so is $g.$ Let $A \subseteq X$ be an arbitrary unbounded subset. To simplify notation, we will define the following:
\begin{equation*}
\begin{split}
D & :=U(U(A,g),g),\\
E & := U(A,g),\\
F & := U(A,f).
\end{split}
\end{equation*}

We wish to show that $D\bar{ {\bf b}}(X\setminus F),$ where ${\bf b}$ is the coarse proximity relation induced by the metric $d.$ For contradiction assume that $D{\bf b}(X\setminus F).$
Then there is an $\epsilon<\infty$ such that for every natural number $n$ there exists $x_{n}\in ((X\setminus F)\setminus B_{f(n)})$ and $y_{n}\in (D\setminus \mathscr{B}_{f(n)})$ such that 
\[d(x_{n},y_{n})<\epsilon.\]
Since $F$ is a coarse neighborhood of $A$ of radius $f,$ by Lemma \ref{lemmaforstarrefinement} we have that for any $n>4$
\[d((X \setminus F) \setminus \mathscr{B}_{f(n-3)}, A \setminus \mathscr{B}_{f(n-3)})>n-4.\]
Find $n$ so large that it satisfies the following inequalities: 
\begin{equation*}
\begin{split}
n & >4,\\
n -4 & > \epsilon + 2 \sqrt{n+2} +3,\\
n & > \lceil \sqrt{n+2} \rceil ,\\
n-1 & > \lceil \sqrt{n+2} \rceil +1.
\end{split}
\end{equation*}
Notice that the above inequalities are satisfied for any $k \geq n.$
Let $k$ be the largest number such that $x_n \notin \mathscr{B}_{f(k)}.$ Then $x_n \in \mathscr{B}_{f(k+1)}.$ Clearly $k \geq n.$ Since $d(x_{n},y_{n})<\epsilon<n \leq k,$ by Lemma \ref{furtherandfurther}, $y_n$ can be in $C_{k+2}, C_{k+1},$ or $C_{k}.$ In particular, $y_n \in \mathscr{B}_{f(k+2)}.$ Therefore $y_{n}$ has to be in $\mathscr{B}_{g( \lceil \sqrt{k+2} \rceil)}.$ Because if it is not, then 
\[y_n \notin \mathscr{B}_{g( \lceil \sqrt{k+2} \rceil)} = \mathscr{B}_{f(\lceil \sqrt{k+2} \rceil^2)} \supseteq \mathscr{B}_{f(k+2)},\]
a contradiction. Thus, since $y_n \in D$ and $y_n \in \mathscr{B}_{g( \lceil \sqrt{k+2} \rceil)},$ by Lemma \ref{lemmalemma} there exists $z \in E,$ such that 
\[d(y_n, z)< \lceil \sqrt{k+2} \rceil.\]
 Since $y_n \in \mathscr{B}_{g( \lceil \sqrt{k+2} \rceil)}$ and $d(y_n, z)< \lceil \sqrt{k+2} \rceil,$ by the proof of the Lemma \ref{furtherandfurther} we have that $z \in \mathscr{B}_{g( \lceil \sqrt{k+2} \rceil +1)}.$ Thus, since $z \in E,$ and $z \in \mathscr{B}_{g( \lceil \sqrt{k+2} \rceil +1)},$ again by Lemma \ref{lemmalemma} there exists $a \in A$ such that 
\[d(z,a)<\lceil \sqrt{k+2} \rceil +1.\]
Let us now examine how close $a$ is to $x_0$. We do it step by step. We know that $y_n$ can be in $C_{k+2}, C_{k+1},$ or $C_{k}.$ In particular, $y_n \notin \mathscr{B}_{f(k-1)}.$ Since $d(y_n, z)< \lceil \sqrt{k+2} \rceil<k,$ we have that $z \notin \mathscr{B}_{f(k-2)}.$ Since $d(z,a)<\lceil \sqrt{k+2} \rceil +1<k-1,$ we have that $a \notin \mathscr{B}_{f(k-3)}.$ So, we have $x_{n} \in ((X\setminus F)\setminus \mathscr{B}_{f(k-3)}),  a \in (A \setminus \mathscr{B}_{f(k-3)}),$ and 
\begin{equation*}
\begin{split}
d(x_n, a) & \leq d(x_n, y_n)+d(y_n,z)+d(z,a)\\
& \leq \epsilon + \lceil \sqrt{k+2} \rceil + \lceil \sqrt{k+2} \rceil +1\\
& \leq \epsilon + \sqrt{k+2}+1 +\sqrt{k+2} +1+1\\
& = \epsilon +2\sqrt{k+2} + 3\\
& < k-4,
\end{split}
\end{equation*}
a contradiction to
\[d((X \setminus F) \setminus \mathscr{B}_{f(k-3)}, A \setminus \mathscr{B}_{f(k-3)})>k-4. \qedhere\]
\end{proof}

\begin{definition}\label{coarsestarrefine}
	Given a metric space $(X,d)$, a point $x_{0}\in X$, and two adequate sequences $f$ and $g$ such that $g$ satisfies the relation in Proposition \ref{starrefinement}, the sequence $g$ is said to be a \textbf{coarse star refinement of $f$ with respect to $x_{0}$}.
\end{definition}

\begin{corollary}\label{corollary13}
	Let $(X,d)$ be a metric space and $x_{0}\in X$. Then given two adequate sequences $f,g:\mathbb{N}\rightarrow\mathbb{R},$ there is another such sequence $h$ such that for all unbounded $A\subseteq X,$ we have that
	\[U(A,h)\ll U(A,f) \quad \text{ and } \quad U(A,h)\ll U(A,g).\]
\end{corollary}

\begin{proof}
By Proposition \ref{starrefinement}, there exist adequate sequences $f_1$ and $g_1$ such that
\[U(U(A,f_1),f_1)\ll U(A,f) \quad  \text{ and } \quad U(U(A,g_1),g_1)\ll U(A,g).\]
In particular, we have that 
\[U(A,f_1) \ll U(A,f) \quad  \text{ and } \quad U(A,g_1) \ll U(A,g).\]
Define $h(n)=\max\{f(n),g(n)\}$. Notice that $h$ is an adequate sequence and by Proposition \ref{basicproperties},
\[U(A,h) \ll U(A,f) \quad  \text{ and } \quad U(A,h) \ll U(A,g). \qedhere\]
\end{proof}

\begin{definition}
	Given a coarse proximity space $(X,\mathcal{B},{\bf b}),$ the {\bf hyperspace at infinity} of $X$, denoted $\mathcal{H}_{\infty}(X)$, is the set $\{A\subseteq X\mid A\notin\mathcal{B}\}$. 
\end{definition}

The following theorem defines a proximity on the hyperspace at infinity.

\begin{theorem}\label{proximityatinfinityversion1}
	Given a metric space $(X,d)$, a point $x_{0}\in X$, and its corresponding hyperspace at infinity $\mathcal{H}_{\infty}(X),$ define a relation $\delta$ on the powerset of $\mathcal{H}_{\infty}(X)$ in the following way:
	
	\vspace{2mm}
	 $\mathcal{A}\delta\mathcal{C}$ if and only if for every adequate sequence $f$ there is an $A\in\mathcal{A}$ and a $C\in\mathcal{C}$ such that $A\ll U(C,f)$ and $C\ll U(A,f)$. 
	 \vspace{2mm}
	 
	\noindent Then $\delta$ is a proximity on $\mathcal{H}_{\infty}(X).$
\end{theorem}

\begin{remark}\label{contrapositiveofcloseness}
Notice that the statement $\mathcal{A} \bar{\delta} \mathcal{C}$ is equivalent to the existence of an adequate sequence $f$ such that for all $A \in \mathcal{A}$ and $C \in \mathcal{C}$, either $A\not\ll U(C,f)$ or $C\not\ll U(A,f)$. Such a sequence $f$ will be called a \textbf{witnessing (adequate) sequence} for $\mathcal{A} \bar{\delta} \mathcal{C}$.
\end{remark}

\begin{proof}[Proof of theorem \ref{proximityatinfinityversion1}]
	The only axioms of a proximity that are not immediate from the definition of $\delta$ are the union and strong axioms. We will show these here.

	{\bf Union axiom:} Assume that $\mathcal{A},\mathcal{C},\mathcal{D}\subseteq\mathcal{H}_{\infty}(X)$ and $(\mathcal{C}\cup\mathcal{D})\delta \mathcal{A}$. Assume towards a contradiction that $\mathcal{C}\bar{\delta}\mathcal{A}$ and $\mathcal{D}\bar{\delta}\mathcal{A}$. Then there are witnessing adequate sequences $f_{1}$ and $f_{2},$ respectively. By Corollary \ref{corollary13} there exists an adequate sequence $g$ such that for all unbounded sets $A\subseteq X$ we have	
	\[U(A,g)\ll U(A,f_{1})\text{ and }U(A,g)\ll U(A,f_{2}).\]
Because $(\mathcal{C}\cup\mathcal{D})\delta \mathcal{A},$ there is some $C\in(\mathcal{C}\cup\mathcal{D})$ and some $A\in\mathcal{A}$ such that $C\ll U(A,g)$ and $A\ll U(C,g).$ If $C\in\mathcal{C},$ then we have 
	\[C\ll U(A,g)\ll U(A,f_{1}) \quad \text{ and } \quad A\ll U(C,g)\ll U(C,f_{1}),\]
a contradiction to $f_1$ being a witnessing cover for $\mathcal{A} \bar{\delta} \mathcal{C}.$ If $C\in\mathcal{D},$ then we have 
	\[C\ll U(A,g)\ll U(A,f_{2}) \quad \text{ and } \quad A\ll U(C,g)\ll U(C,f_{2}),\]
a contradiction to $f_2$ being a witnessing cover for $\mathcal{A} \bar{\delta} \mathcal{D}.$
The converse direction of the union axiom is trivial.
	
	{\bf Strong Axiom:} Let $\mathcal{A},\mathcal{C}\subseteq\mathcal{H}_{\infty}(X)$ be such that $\mathcal{A}\bar{\delta}\mathcal{C}$. Then there exists a witnessing adequate sequence $f$ such that for all $A\in\mathcal{A}$ and all $C\in\mathcal{C}$ one either has $A\not\ll U(C,f)$ or $C\not\ll U(A,f)$. Let $g$ be an adequate sequence such that for all $A\in\mathcal{H}_{\infty}(X)$ we have
\[U(U(A,g),g)\ll U(A,f).\]
Define
\[\mathcal{E}=\{K\in\mathcal{H}_{\infty}(X)\mid \exists C\in\mathcal{C},\,C\ll U(K,g)\ll U(C,f)\}.\]
We claim that $\mathcal{A}\bar{\delta}\mathcal{E}$ and $(\mathcal{H}_{\infty}(X)\setminus\mathcal{E})\bar{\delta}\mathcal{C}$. If $\mathcal{A}\delta\mathcal{E},$ then there is some $A\in\mathcal{A}$ and some $K\in\mathcal{E}$ such that $A\ll U(K,g)$ and $K\ll U(A,g)$. Let $C$ be a member of $\mathcal{C}$ that witnesses $K$ being a member of $\mathcal{E}$. Then $ U(K,g)\ll U(C,f),$ which implies $A\ll U(C,f)$. Also, by (vi) of Proposition \ref{basicproperties}, we have that $K\ll U(A,g)$ implies that there is a bounded set $D$ such that $U(K,g)\setminus D\subseteq U(U(A,g),g)\ll U(A,f).$ Thus, $U(K,g)\ll U(A,f),$ and hence $C\ll U(A,f).$ Therefore, we have  $A\ll U(C,f)$ and $C\ll U(A,f),$ which is a contradiction to $f$ being a witnessing sequence. Therefore $\mathcal{A}\bar{\delta}\mathcal{E}$.
	
	Now assume towards a contradiction that $(\mathcal{H}_{\infty}(X)\setminus\mathcal{E})\delta\mathcal{C}$. Then let $K\in(\mathcal{H}_{\infty}(X))\setminus\mathcal{E}$ and $C\in\mathcal{C}$ be such that $K\ll U(C,g)$ and $C\ll U(K,g)$. The first of these implies that there is a bounded set $D$ such that $U(K,g)\setminus D\subseteq U(U(C,g),g)\ll U(C,f)$, which in turn implies that $U(K,g)\ll U(C,f)$. However this implies that $K\in\mathcal{E},$ which is a contradiction. Therefore $(\mathcal{H}_{\infty}(X)\setminus\mathcal{E})\bar{\delta}\mathcal{C}$ which established the strong axiom for $\delta$.
\end{proof}

Note that the coarse neighborhoods $U(A,f)$ and hence the proximity $\delta$ are defined with respect to a particular point $x_{0}$ within our metric space $X$. Proposition \ref{metricbase} showed that regardless of the choice of point $x_{0},$ the resulting coarse neighborhoods of the form $U(A,f)$ will make up a coarse neighborhood base at any subset $A$ of $X$. Now we will show that the proximity on the hyperspace also does not depend on the choice of the base point. For the sake of clarity, we will return to our previous abbreviated notation involving the basepoint, i.e., $U_{x_0}(A,f).$

\begin{lemma}\label{technicallemma}
Let $(X,d)$ be a metric space and $x_{0},x_{1}$ distinct points of $X.$ Then for any adequate sequence $f,$ there exists an adequate sequence $g$ such that for all subsets $C \subseteq X,$
\[U_{x_{0}}(C,g)\ll U_{x_{1}}(C,f).\]
\end{lemma}
\begin{proof}
Let $f$ be an adequate sequence. Without loss of generality we can assume that $f$ takes integer values. We define an adequate sequence $g$ in the following way: for each $n\in\mathbb{N},$ there is a least natural number $T(n)$ such that $B(x_{1},f(n))\subseteq B(x_{0},T(n)).$ Define $g: \mathbb{N} \to \mathbb{R}$ by setting $g(1)=T(1)$ and then inductively by
\[g(n):=\max\{T(n^2), g(n-1)+n+2\}.\]
Notice that the second condition implies that $g$ is an adequate sequence. We then claim that for all subsets $C\subseteq X$ we have
\[U_{x_{0}}(C,g)\ll U_{x_{1}}(C,f).\]
	Denote the set on the left hand side by $D$ and the set on the right hand side by $E$. Assume towards a contradiction that $D{\bf b}(X\setminus E),$ where ${\bf b}$ is the coarse proximity induced by the metric. Then there is an $\epsilon<\infty$ such that for every $n\in\mathbb{N}$ there is
\[x_{n}\in (X\setminus E)\setminus \mathscr{B}(x_{0},g(n))\quad \text{ and }y_{n} \quad \in D\setminus \mathscr{B}(x_{0},g(n))\]
such that $d(x_{n},y_{n})<\epsilon$. Let $k_n$ be the greatest natural number such that $x_{n}\notin \mathscr{B}(x_{0},g(k_n))$. Then $x_{n}\in \mathscr{B}(x_{0},g(k_n+1))$. By Lemma \ref{furtherandfurther} we have that for any $n>\epsilon,$ $y_{n}\notin \mathscr{B}(x_{0},g(k_n-1))$ and $y_{n}\in \mathscr{B}(x_{0},g(k_n+2))$. Then, by Lemma \ref{lemmalemma} we have that there must be a $c_{n}\in C$ such that $d(y_{n},c_n)<k_n+2$. Notice that since $y_{n}\notin \mathscr{B}(x_{0},g(k_n-1))$ and $d(y_{n},c_n)<k_n+2,$ $g$ being is an adequate sequence implies $c_n \notin \mathscr{B}(x_0, g(k_n-3))$. Also, by the triangle inequality we have that for all $n> \epsilon,$
	\[d(x_{n},c_n)<\epsilon+k_n+2.\]
	However, for all $n> \epsilon$ we also have that
	\[x_{n},c_{n}\notin \mathscr{B}(x_{0},g(k_{n}-3)) \supseteq \mathscr{B}(x_{0},T((k_{n}-3)^{2}))\supseteq \mathscr{B}(x_{1},f((k_{n}-3)^{2})).\]
	Thus, by Lemma \ref{lemmaforstarrefinement} we have that $d(x_{n},c_{n})>(k_{n}-3)^{2}$ for all $n> \epsilon.$ But for large enough $n,$ this contradicts $d(x_{n},c_{n})<\epsilon+k_{n}+2$. Therefore, it has to be that $U_{x_{0}}(C,g)\ll U_{x_{1}}(C,f).$
\end{proof}

\begin{theorem}\label{basepointindependence}
	Let $(X,d)$ be a metric space and $x_{0},x_{1}$ distinct points of $X$. If $\delta_{0}$ and $\delta_{1}$ are the respective proximities on $\mathcal{H}_{\infty}(X)$ constructed using $x_{0}$ and $x_{1}$ as in Theorem \ref{proximityatinfinityversion1}, then the proximity relations $\delta_{0}$ and $\delta_{1}$ are equal.
\end{theorem}
\begin{proof}
Assume $\mathcal{A},\mathcal{C}\subseteq\mathcal{H}_{\infty}(X)$ be such that $\mathcal{A}\delta_{0}\mathcal{C}$. Let $f$ be an arbitrary adequate sequence. Then by Lemma \ref{technicallemma}, there exists an adequate sequence $g$ such that for all subsets $C \subseteq X,$
\[U_{x_{0}}(C,g)\ll U_{x_{1}}(C,f).\]
Since $\mathcal{A}\delta_{0}\mathcal{C},$ there exists $A\in\mathcal{A}$ and a $C\in\mathcal{C}$ such that 
\[A\ll U_{x_{0}}(C,g) \quad \text{ and } \quad C\ll U_{x_{0}}(A,g),\]
which by the property of $g$ gives us 
\[A\ll U_{x_{0}}(C,g)\ll U_{x_{1}}(C,f)\quad \text{ and }\quad C\ll U_{x_{0}}(A,g)\ll U_{x_{1}}(A,f).\]
Thus, we have $\mathcal{A}\delta_{1}\mathcal{C}$. Similarly one can show that $\mathcal{A}\delta_{1}\mathcal{C}$ implies $\mathcal{A}\delta_{0}\mathcal{C}$. Therefore $\delta_{0}=\delta_{1}$.	
\end{proof}

\begin{definition}
	Let $(X,d)$ be a metric space and $(\mathcal{H}_{\infty}(X),\delta_{0})$ the corresponding proximity space (constructed with respect to some point $x_{0}\in X$). Define the set ${\bf B}X$ to be the set of all $\phi$ equivalence classes of unbounded sets in $X,$ where $\phi$ is the weak asymptotic resemblance induced by the coarse proximity induced by $d.$ By Proposition \ref{metricconsistency} this relation is equivalent to the relation of having finite Hausdorff distance. Endow ${\bf B}X$ with quotient proximity $\bm{\delta}$ induced by the projection $\pi:(\mathcal{H}_{\infty}(X),\delta_{0})\rightarrow{\bf B}X,$ as in Definition \ref{surjectiveproximity}. The quotient proximity space $({\bf B}X,\bm{\delta})$ is called the {\bf proximity space at infinity} of $X$.
\end{definition}

\begin{remark}
If $A$ is a subset of $X,$ then $\phi$ equivalence class of $A$ (i.e. a point in ${\bf B} X$) will be denoted by $[A].$
\end{remark}

Our goal is to show that the proximity space at infinity induces a functor from the category of unbounded metric spaces whose morphisms are close equivalence classes of coarse proximity maps to the category of proximity spaces whose morphisms are proximity maps.

\begin{theorem}\label{inducedmap}
	Let $(X,d_{0}),(Y,d_{1})$ be unbounded metric spaces, $h:X\rightarrow Y$ a coarse proximity map, and $(\bm{B}X,\bm{\delta}_{0}),(\bm{B}Y,\bm{\delta}_{1})$ the corresponding proximity spaces at infinity. Then the map $\bm{B}h:\bm{B}X\rightarrow\bm{B}Y$ defined by
	\[\bm{B}h([A])=[h(A)]\]
	is a well-defined proximity map. Moreover, if $l:X\rightarrow Y$ is a coarse proximity map that is close to $h,$ then $\bm{B}h=\bm{B}l$.
\end{theorem}
\begin{proof}
	The well-definedness of $\bm{B}h$ follows from Proposition \ref{proximitymapspreserveweakasymptoticresemblance}. The equality of $\bm{B}h$ and $\bm{B}l$ for close coarse proximity maps $h$ and $l$ follows from the definition of closeness of coarse proximity maps. Let us show that ${\bf B}h$ is a proximity map. Let $x_{0}\in X$ and $y_{0}=h(x_{0})$. Let $\delta_{0}$ be the proximity on $\mathcal{H}_{\infty}(X)$ constructed using the basepoint $x_{0}$ and let $\delta_{1}$ be proximity on $\mathcal{H}_{\infty}(Y)$ constructed using the point $y_{0}=h(x_{0})$. We then consider the following commutative diagram:
	\begin{center}
		\begin{tikzcd}[column sep=large]
			(\mathcal{H}_{\infty}(X),\delta_{0}) \arrow{r}{h} \arrow{dd}{\pi_{X}} \arrow{rdd}{{\bf B}h\circ\pi} 
			&  (\mathcal{H}_{\infty}(Y),\delta_{1}) \arrow{dd}{\pi_{Y}}\\
			& \\
			({\bf B}X,\bm{\delta}_{0}) \arrow{r}{{\bf B}h} & ({\bf B}Y,\bm{\delta}_{1})
		\end{tikzcd}
	\end{center}
where $h$ is the obvious induced map on the hyperspaces. Notice that since the diagram is commutative, the map ${\bf B}h\circ\pi$ is well-defined.
	By Proposition \ref{quotientproximityproperty} the function ${\bf B}h$ is a proximity map if and only if the function ${\bf B}h\circ\pi$ is a proximity map. To show that ${\bf B}h\circ\pi$ is a proximity map, it is enough to show that $h$ is a proximity map (${\bf B}h\circ\pi$ is then a composition of two proximity maps, and therefore a proximity map). Let $\mathcal{A},\mathcal{C}\subseteq\mathcal{H}_{\infty}(X)$ be such that $\mathcal{A}\delta_{0}\mathcal{C}$. We will show that $h(\mathcal{A})\delta_{1}h(\mathcal{C}),$ which will complete our proof. Let $f_{1}$ be an adequate sequence and $f_{2}$ an adequate sequence that coarse star refines $f_{1}$. Since $h$ is a coarse proximity map, it is proper. Thus, for every $n\in\mathbb{N}$ there is a least $k\in\mathbb{N}$ such that
	\[h^{-1}(\mathscr{B}(y_{0},f_{2}(n)))\subseteq \mathscr{B}(x_{0},k).\]
	We will denote this natural number by $T(n)$. Likewise, $h$ is bornologous, so for every $n\in\mathbb{N}$ there is a greatest natural number $m$ (possibly also $\infty$ for the first few $n$'s) such that
	\[d(x,y)\leq m\implies d(h(x),h(y))<n.\]
	We will denote this number by $\rho(n)$ (if $\rho(n)= \infty$ for some $n$, then set $\rho(n)=1$ instead). Since $X$ and $Y$ are unbounded and $f$ is a coarse proximity map, the functions $\rho$ and $T$ as sequences must be nondecreasing and divergent. We can choose a sequence $(n_{k})$ of natural numbers such that for any $k \in \mathbb{N},$ the following conditions hold:
	\begin{enumerate}[(i)]
	\item $k < T(n_{k}),$
	\item $k+1 < \rho(n_{k}),$
	\item $\max\{T(n_{k}),\rho(n_{k})\}-\max\{T(n_{k-1}),\rho(n_{k-1})\}>k+1.$
	\end{enumerate}
 We then define an adequate sequence $g$ by
	\[g(k)=\max\{T(n_{k}),\rho(n_{k})\}.\]
	Because $\mathcal{A}\delta_{0}\mathcal{C}$ we have that there is an $A\in\mathcal{A}$ and a $C\in\mathcal{C}$ such that $A\ll U_{x_{0}}(C,g)$ and $C\ll U_{x_{0}}(A,g)$. We then claim the following:
	\[h(A)\ll U_{y_{0}}(h(C),f_1)\text{ and }h(C)\ll U_{y_{0}}(h(A),f_{1})\]
	We will show the first of these. The second is shown similarly. Let $x\in A\ll U_{x_{0}}(C,g)$. Then there is a greatest integer $k$ such that $x\notin \mathscr{B}(x_{0},g(k))$. Then $x\in \mathscr{B}(x_{0},g(k+1))$. This implies that there is a $c\in C$ such that $d(x,c)<k+1$. Since $x\notin \mathscr{B}(x_{0},g(k)),$ we have that $x\notin \mathscr{B}(x_{0},T(n_{k}))\cup \mathscr{B}(x_{0},\rho(n_{k}))$. This implies that $h(x)\notin \mathscr{B}(y_{0},f_2(n_{k}))$. Likewise, because $k+1<\rho(n_{k})$ we have that $d(h(x),h(c))<n_{k}$. Therefore we have that $h(x)\in h(C)_{n_{k}}^{f_{2}}$ and hence, up to a bounded set, $h(A)\subseteq U_{y_{0}}(h(C),f_{2})$. Then, because $f_{2}$ coarse star refines $f_{1}$ we have $h(A)\ll U_{y_{0}}(h(C),f_{1})$. Similarly $h(C)\ll U_{y_{0}}(h(A),f_{1})$. Thus, $h(\mathcal{A})\delta_{1}h(\mathcal{C}),$ which establishes that $h:(\mathcal{H}_{\infty}(X),\delta_{0})\rightarrow(\mathcal{H}_{\infty}(Y),\delta_{1})$ is a proximity map, which consequently implies that ${\bf B}h$ is a proximity map.
\end{proof}

\begin{corollary}
	The assignment of the proximity space $(\bm{B}X,\bm{\delta})$ to an unbounded metric space $(X,d)$ and the assignment of $\bm{B}f:\bm{B}X\rightarrow\bm{B}Y$ to a closeness equivalence class of coarse proximity maps $[f]:X\rightarrow Y$ between unbounded metric spaces makes up a functor $\bm{B}$ from the category of unbounded metric spaces whose morphisms are close equivalence classes of coarse proximity maps to the category of proximity spaces whose morphisms are proximity maps. \hfill $\square$
\end{corollary}

\begin{corollary}
	If $(X,d_{1})$ and $(Y,d_{2})$ are unbounded proximally coarse equivalent metric spaces, then their corresponding proximity spaces at infinity are proximally isomorphic. In particular, they are homeomorphic. \hfill $\square$
\end{corollary}

\begin{example}
	Let $X=\{n^{2}\mid n\in\mathbb{N}\}\cup\{0\}$ be equipped with its usual metric. Then if $A,B\subseteq X$ are unbounded subsets we have that the Hausdorff distance between $A$ and $B$ is finite if and only if $A$ and $B$ differ by a bounded set. Likewise, there is an adequate sequence $g$ such that for all unbounded sets $A$ one has that $U_{0}(A,g)\setminus A$ is bounded (one could take $g(n)=n^{3}$ for example). Then if $\mathcal{H}_{\infty}(X)$ is given the proximity $\delta$ constructed using the basepoint $0,$ we have that two subsets $\mathcal{A},\mathcal{B}\subseteq\mathcal{H}_{\infty}(X)$ are close if any only if there is an $A\in\mathcal{A}$ and a $B\in\mathcal{B}$ such that the Hausdorff distance between $A$ and $B$ is finite. Thus the proximity $\hat{\bm{\delta}}$ on ${\bf B}X$ defined by $\pi_{X}(\mathcal{A})\hat{\bm{\delta}}\pi_{X}(\mathcal{B})$ if and only if $\pi_{X}(\mathcal{A})\cap\pi_{X}(\mathcal{B})\neq\emptyset$ is a proximity on ${\bf B}X$ for which the projection $\pi_{X}:X\rightarrow{\bf B}X$ is a proximity map. The proximity $\hat{\bm{\delta}}$ is the finest possible proximity on ${\bf B}X,$ and hence the finest proximity on ${\bf B}X$ for which $\pi_{X}$ is a proximity map. Thus, $\hat{\bm{\delta}}$ is the quotient proximity and $({\bf B}X,\hat{\bm{\delta}})$ is the proximity at infinity of $X$. The topology on ${\bf B}X$ is discrete.
\end{example}

\begin{proposition}
	Let $\mathbb{Z}$ have its natural metric structure and corresponding coarse proximity structure. Then ${\bf B}\mathbb{Z}$ is not connected and has at least $3$ connected components.
\end{proposition}
\begin{proof}
	Consider the following subsets of $\mathcal{H}_{\infty}(\mathbb{Z})$
	
	\[\begin{array}{rcl}
	\mathcal{A} & := & \{A\in\mathcal{H}_{\infty}(\mathbb{Z})\mid \exists z\in\mathbb{Z}\,\forall x\in A\,z\leq x\},\\
	\mathcal{C} & := & \{C\in\mathcal{H}_{\infty}(\mathbb{Z})\mid\exists z\in\mathbb{Z}\,\forall x\in C\,z\geq x\},\\
	\mathcal{D} & := & \mathcal{H}_{\infty}(\mathbb{Z})\setminus(\mathcal{A}\cup\mathcal{C}),
	\end{array}\]
	
	i.e., $\mathcal{A}$ is the set of unbounded subsets of $\mathbb{Z}$ that have a lower bound, $\mathcal{C}$ is the set of unbounded subsets of $\mathbb{Z}$ that have an upper bound, and $\mathcal{D}$ is the set of unbounded subsets of $\mathbb{Z}$ that have neither a lower bound nor an upper bound. Clearly $\mathcal{A}, \mathcal{C},$ and $\mathcal{D}$ are mutually disjoint. These three sets are trivially closed under the relation of having finite Hausdorff distance, i.e., if $A,B \in \mathcal{H}_{\infty}(\mathbb{Z})$ and $d_H(A,B)< \infty,$ then both $A$ and $B$ are in $\mathcal{A},$ both $A$ and $B$ are in $\mathcal{C},$ or both $A$ and $B$ are in $\mathcal{D}.$ If $f$ is any adequate sequence, then given $A\in\mathcal{A},$ there is no $C\in\mathcal{C}$ or $D\in\mathcal{D}$ such that $D\ll U(A,f)$ or $C\ll U(A,f),$ regardless of the choice of a basepoint. Similarly,  if $f$ is any adequate sequence, then given $C\in\mathcal{C},$ there is no $A\in\mathcal{A}$ or $D\in\mathcal{D}$ such that $D\ll U(C,f)$ or $A\ll U(C,f),$ regardless of the choice of a basepoint. Consequently, no two of $\mathcal{A},\mathcal{C},$ or $\mathcal{D}$ are close in the hyperspace at infinity. We then let $\bm{3}=\{a,c,d\}$ be the discrete proximity space on $3$ elements. The function $h:(\mathcal{H}_{\infty}(\mathbb{Z}),\delta)\rightarrow\bm{3}$ defined by $h(\mathcal{A})=a,\,h(\mathcal{C})=c,$ and $h(\mathcal{D})=d$ is a proximity mapping that is constant on the fibers of the projection $\pi:\mathcal{H}_{\infty}(\mathbb{Z})\rightarrow{\bf B}\mathbb{Z}$. Thus there is a unique proximity map $g:{\bf B}\mathbb{Z}\rightarrow\bm{3}$ such that $g\circ\pi=h$. Thus ${\bf B}\mathbb{Z}$ is not connected and has at least $3$ connected components.
\end{proof}

\begin{proposition}
	Let $\mathbb{N}$ have its natural metric structure and corresponding coarse proximity structure. Then the proximity on ${\bf B}\mathbb{N}$ is not discrete. 
\end{proposition}
\begin{proof}
Let $\mathcal{A}\subseteq\mathcal{H}_{\infty}(\mathbb{N})$ be the set of all unbounded subsets of $\mathbb{N}$ that have asymptotic dimension $0$. Let $\mathcal{C}\subseteq\mathcal{H}_{\infty}(\mathbb{N})$ be the set of all unbounded subsets of $\mathbb{N}$ that have asymptotic dimension $1$. It is clear that $\mathcal{A}$ and $\mathcal{C}$ are closed under the relation of having finite Hausdorff distance. It is also clear that $\mathcal{A}$ and $\mathcal{C}$ are disjoint. We will show that $\mathcal{A}\delta\mathcal{C},$ which will imply that $\pi(\mathcal{A})\bm{\delta}\pi(\mathcal{C}),$ showing that the proximity on ${\bf B}\mathbb{N}$ is not discrete.

We will use $1$ as a basepoint. Let $f$ be an adequate sequence. Let us first construct an unbounded set $A$ of asymptotic dimension $0$ in the following way: define
\[g(n) = \lceil f(100n)\rceil,\]
and define $A_{1}$ to be the integral interval $[1,g(1)]$. Then, for every natural number $n > 1$ let $\eta_{n}=\max(A_{n-1}),$ and define 
\[A_{n}=\{m\in\mathbb{N}\mid\exists k\in\mathbb{N}\cup\{0\},\, m=\eta_{n}+kn,\,m\leq g(n)\}.\]
Finally, define $A=\bigcup A_{n}$. This set is clearly unbounded and has asymptotic dimension $0$ because for each real number $r\geq 0$ the set of $r$-components is uniformly bounded. Also, by construction of $A$ we have that $U_{1}(A,f)=\mathbb{N},$ and consequently $U_{1}(A,f)$ is a coarse neighborhood of every subset of $\mathbb{N}$. Thus, setting $C=\mathbb{N},$ we have that $A \in \mathcal{A}, C \in \mathcal {C}, A\ll U_1(C,f),$ and $C\ll U_1(A,f).$ Since $f$ was arbitrary, this shows that $\mathcal{A}\delta\mathcal{C},$ which consequently implies that $\pi(\mathcal{A})\bm{\delta}\pi(\mathcal{C}),$ as desired.
\end{proof}

\begin{remark}
The above proof shows that as an element of the hyperspace at infinity of $\mathbb{N},$ the singleton $\mathbb{N}$ is close to the set consisting of asymptotic dimension $0$ sets.
\end{remark}

\begin{corollary}
	If $(X,d)$ is an unbounded metric space into which $\mathbb{N}$ coarsely embeds, then ${\bf B}X$ is not discrete.
\end{corollary}

\section{Questions}

We conclude with some open questions.
\vspace{\baselineskip}

Intuitively we think of unbounded subsets in a proper metric space $X$ as corresponding to their trace in the Higson corona of $X$. The proximity space $\bm{B}X$ was motivated by considering the Vietoris topology on the hyperspace, $\mathcal{H}(\nu X)$, of the Higson corona. The natural question is then the following:

\begin{question}
	Given an unbounded proper metric space $(X,d)$ what relationship exists between the Vietoris topology on the hyperspace of the Higson corona, $\nu X$, and the proximity space at infinity of $X$, ${\bf B}X$?
\end{question}

It has been proven that if a proper metric space $X$ has finite asymptotic dimension equal to $n,$ then the Higson corona of that space has topological dimension equal to $n$. The question of whether or not the topological dimension of the Higson corona of a proper metric space of infinite asymptotic dimension is infinite remains open. We consider a similar relationship between dimensions in the following question.

\begin{question}
	Given an unbounded metric space $(X,d)$ what is the relationship between the topological dimension of $\bm{B}X$ and the asymptotic dimension of $X$?
\end{question}

\bibliographystyle{abbrv}
\bibliography{Coarse_Proximity_and_Proximity_at_Infinity}{}

\end{document}